\documentclass[a4paper, 12pt,reqno]{amsart} 
\usepackage{amssymb}
\usepackage{amsmath}

\newtheorem{thm}{Theorem}[section]
\newtheorem{prop}{Proposition}[section]
\newtheorem{lem}{Lemma}[section]
\newtheorem{rem}{Remark}[section]
\newtheorem{cor}{Corollary}[section]
\newtheorem{expl}{Example}[section]
\newtheorem{defn}{Definition}[section]

\makeatletter
   
          \@addtoreset{equation}{section}
\makeatother

\begin{document}

\title[Dirac operators]{Eigenfunctions at the threshold energies 
of magnetic Dirac operators
}

\thanks{${}^{\ddag}$ {\scriptsize Supported by 
     Grant-in-Aid for Scientific Research (C) 
    No.  21540193, 
   Japan Society for the Promotion of Science.}}

\maketitle

\centerline{\author{Yoshimi Sait\={o}${}^{\dag}$ and 
Tomio Umeda${}^{\ddag}$}}

\vspace{6pt}

\small{
\centerline{${}^{\dag}$\it Department of Mathematics, University of Alabama at Birmingham} 
 
\centerline{\it Birmingham, AL 35294, USA}
 
\vspace{6pt}

\centerline{${}^{\ddag}$\it Department of Mathematical Sciences,
University of Hyogo}
\centerline{\it Himeji 671-2201, Japan}

\vspace{6pt}
\centerline{${}^{\dag}$\it saito@math.uab.edu}
\vspace{4pt}
\centerline{${}^{\ddag}$\it umeda@sci.u-hyogo.ac.jp}
}

\vspace{12pt}

\begin{quote}
{\small 
Discussed are   $\pm m$ modes and $\pm m$ resonances 
of Dirac operators with vector potentials
$H_{\!A}= \alpha \cdot (D - A(x) ) + m \beta$.
Asymptotic limits  
of $\pm m$ modes at infinity 
are derived  when $|A(x)|\le C\langle x\rangle^{-\rho}$, $\rho > 1$, 
provided that 
$H_A$ has $\pm m$ modes.
In wider classes of 
vector potentials,
sparseness of the  vector
potentials which give rise to the $\pm m$ modes of $H_A$
are established.
It is proved that no $H_A$ has $\pm m$ resonances
if
$|A(x)|\le C\langle x\rangle^{-\rho}$, $\rho >3/2$.

\vspace{8pt}

\vspace{8pt}

}
\end{quote}

\vspace{12pt}
\section{Introduction}

The introduction is devoted to exhibiting our results as well as 
to reviewing previous contributions in connection with 
the results in the present paper.

This paper 
is concerned with eigenfunctions and 
resonances  at the threshold energies
of Dirac operators with vector potentials
\begin{equation} \label{eqn:1-1}
H_{\!A}= \alpha \cdot (D - A(x) ) + m \beta , \quad D=\frac{1}{\, i \,}
\nabla_x,
\,\,\, x \in {\mathbb R}^3.
\end{equation}
Here $\alpha= (\alpha_1, \, \alpha_2, \, \alpha_3)$ is
the triple of  $4 \times 4$ Dirac matrices
\begin{equation*}\label{eqn:1-2}
\alpha_j = 
\begin{pmatrix}
 \mathbf 0 &\sigma_j \\ \sigma_j &   \mathbf 0 
\end{pmatrix}  \qquad  (j = 1, \, 2, \, 3)
\end{equation*}
with  the $2\times 2$ zero matrix $\mathbf 0$ and
the triple of  $2 \times 2$ Pauli matrices
\begin{equation*}\label{eqn:1-3}
\sigma_1 =
\begin{pmatrix}
0&1 \\ 1& 0
\end{pmatrix}, \,\,\,
\sigma_2 =
\begin{pmatrix}
0& -i  \\ i&0
\end{pmatrix}, \,\,\,
\sigma_3 =
\begin{pmatrix}
1&0 \\ 0&-1
\end{pmatrix},
\end{equation*}
and
\begin{equation*}\label{eqn:1-4}
\beta=
\begin{pmatrix}
I_2& \mathbf 0 \\ \mathbf 0 & -I_2
\end{pmatrix}.
\end{equation*}
The constant $m$ is assumed to be positive.

Throughout the present paper we assume that  each component
of the vector potential $A(x)=(A_1(x), \,A_2(x), \, A_3(x))$ 
is a real-valued 
measurable function. In
addition to this, we shall later impose four different sets of
 assumptions on
$A(x)$ under which the operator $-\alpha \cdot A(x)$
is relatively compact with respect to the free Dirac operator 
$H_0=\alpha \cdot D  + m \beta $. 
Therefore, under any set of assumptions to be made, 
the  magnetic Dirac operator 
$H_{\!A}$ is a self-adjoint operator in 
the Hilbert space $[L^2({\mathbb R}^3)]^4$, and
the essential spectrum 
of $H_{\!A}$ is given by the union of the intervals 
$(-\infty, \, -m]$ and $[m, \, +\infty)$:
\begin{equation} \label{eqn:1-5}
\sigma_{\rm{ess}}(H_{\!A}) =  (-\infty, \, -m] \cup  [m, \, +\infty) .
\end{equation}
By the {\it threshold energies} of $H_{\!A}$, we mean  the values $\pm m$, 
the edges of the
essential spectrum 
$\sigma_{\rm{ess}}(H_{\!A})$.

We shall see  in sections \ref{sec:AsymptoticLimits} -- \ref{sec:Structure}
 that
 the discrete spectrum of $H_{\!A}$ 
in the gap $(-m, \, m)$ is empty,
although we should like to mention that
this fact is well-known 
by the result of
Thaller
\cite[p. 195, Theorem 7.1]{Thaller} where smoothness of vector
potentials is assumed though. 
In other words,
there are no isolated eigenvalues with finite  multiplicity
 in the 
spectral gap $(-m, \, m)$.
In the present paper, this fact will be obtained as a by-product
 of 
Theorem \ref{thm:SSDO-main3} 
 in section \ref{sec:Supersymmetric}, 
where we shall deal with an abstract Dirac operator,
\textit{i.e.}, a supersymmetric Dirac operator. 
As a result, we shall  have 
\begin{equation*}
\sigma(H_{\!A})=
\sigma_{\rm{ess}}(H_{\!A}) =  (-\infty, \, -m] \cup  [m, \, +\infty) 
\end{equation*}
under any set of the assumptions on $A(x)$
of the present paper.

In relation with the relative compactness of $-\alpha \cdot A(x)$
with respect to $H_0$,
it is worthwhile  to mention a work by Thaller
\cite{Thaller1}, where he showed that (\ref{eqn:1-5}) is true under the
assumption that $|B(x)| \to 0$ as $|x| \to \infty$. Here
$B(x)$ denotes the magnetic field: $B(x)=\nabla \times A(x)$.
It is clear that 
the assumption that  $|B(x)| \to 0$ does not necessarily 
imply the relative compactness of $-\alpha \cdot A(x)$
with respect to $H_0$.
In Helffer, Nourrigat and Wang \cite{HelfferNourigatWang},
they showed that  (\ref{eqn:1-5}) is true under  much weaker
assumptions on $B(x)$, which do not even need the requirement 
that 
$|B(x)| \to 0$ as $|x| \to \infty$; see also 
\cite[\S 7.3.2]{Thaller}.

It is generally expected
 that eigenfunctions corresponding to
a discrete eigenvalue of $H_A$ decay exponentially at infinity (describing
bound states), and that (generalized) eigenfunctions corresponding to an
energy inside the continuous spectrum $(-\infty, \, -m] \cup  [m, \, +\infty)$
behave like the sum of a plane  
and a spherical waves at infinity (describing  scattering states).
At the energies $\pm m$, on which we shall focus in the present paper,
(generalized) eigenfunctions are expected to behave
like $C_0 + C_1 |x|^{-1} + C_2 |x|^{-2}$ at infinity,
where $C_j$, $j=1$, $2$, $3$, are constant vectors in ${\mathbb C}^4$.
If $C_0=0$ and $C_1\not=0$, then the (generalized) eigenfunctios 
become either of $\pm m$ resonances, and if  
$C_0=0$ and $C_1=0$, then the (generalized) eigenfunctios 
become either of $\pm m$ modes.
For the precise definitions of $\pm m$ resonances and 
$\pm m$ modes,  
see Definition \ref{defn:Mresonance} in section \ref{sec:Resonances}
and
Definition \ref{defn:Mmodes} in section \ref{sec:AsymptoticLimits}
respectively.
 As for the exponential decay of
eigenfunctions, we refer 
the reader to 
works by Helffer and Parisse \cite{HelfferParisse}, 
 Wang \cite{Wang}, and a recent work by Yafaev \cite{Yafaev}.
As for the generalized eigenfunctions 
corresponding to an
energy in $(-\infty, \, -m) \cup  (m, \, +\infty)$ we refer the reader to 
 Yamada \cite{Yamada-2}.

As mentioned above, our main concern is 
the threshold energies $\pm m$ of 
the magnetic Dirac operator $H_A$. 
These energies are
 of particular importance and of interest from
the physics point of  view. 
We
should like to mention  Pickl and D\"urr \cite{PicklDurr},
 and Pickl \cite{Pickl},
where they investigate generalized eigenfunctions 
not only at the energies $\pm m$ 
but also at the energies near
$\pm m$,
with the emphasis on the famous relativistic effect
of the pair creation of an electron and a positron. 
It is worthwhile to note that 
$\pm m$ modes and $\pm m$ resonances  play decisive
roles in their results.
In the same spirit as in 
  \cite{Pickl} and \cite{PicklDurr},
 Pickl and D\"urr \cite{PicklDurr2}  
 mention the possibility of  experimental 
   verifications
  of the pair creation by
  combining lasers and heavy ions fields. 
  Therefore, it is obvious that
  results on $\pm m$ modes and $\pm m$ resonances
  of magnetic Dirac operators 
  are useful to understand
  the physics of the pair creations in such laser fields;
  see \cite{PicklDurr2} for details.

The goal of the present paper
is to derive a series of new results on
 $\pm m$ modes and  $\pm m$ resonances of
    the magnetic Dirac operators $H_{\!A}$.
Precisely speaking, we shall study
 asymptotic behaviors at
infinity  of the $\pm m$ modes, 
show sparseness of vector 
  potentials which give rise to the $\pm m$ modes,
  and establish
  non-existence of $\pm m$ resonances.

According to Pickl \cite[Theorems 3.4, 3.5]{Pickl}, 
the behavior of the generalized eigenfuntions of Dirac operators  near criticality 
largely depends on whether Dirac operators with critical potentials have
 $\pm m$ resonances or not. 
 Since the modulus of their cirtical  potentials
 are less than or equal to
  $C (1+|x|)^{-2}$,  we can actually conclude 
 from Theorems \ref{thm:NER-1} and \ref{thm:NER-2} 
 in section \ref{sec:Resonances}  that the
 magnetic Dirac operators with the critical potentials have no
  $\pm m$ resonances.
However, one has to pay attention to
a slight difference between our definition of 
the threshold resonances (cf. Definition \ref{defn:Mresonance})
and theirs (cf. \cite[Definition 2.3 and the paragraph after it]{Pickl}).

Finally, we would like to mention that 
there is a striking difference between two and three dimensional
Dirac operators with magnetic fields
 at the threshold energies. Compare the results 
in section \ref{sec:Resonances} of the present paper with
those of Aharonov and Casher
\cite{AharonovCasher}, where their arguments indicate that
one can find magnetic Dirac operators in dimension two
which possess threshold resonances.
See also Weidl \cite[section 10]{Weidl}.

The plan of the paper is as follows.
In section \ref{sec:Supersymmetric}, 
 we shall prepare a few results on
 a supersymmetric Dirac 
operator, which will be used in all the later sections.
In section \ref{sec:AsymptoticLimits} we shall investigate
asymptotic behaviors at infinity of $\pm m$ modes
of $H_A$, provided that it has $\pm m$ modes.
Sparseness of the set of  
vector potentials
$A(x)$ which yield $\pm m$ modes of  $H_{\!A}$ will be discussed
in sections \ref{sec:Sparseness} and \ref{sec:Structure} in different regimes.
In section \ref{sec:Resonances} we shall prove that any $H_A$ does not have
$\pm m$ resonances under a stronger assumption 
than those made in the previous  sections. 
Finally in section  \ref{sec:Concludings}
 we shall give examples of vector potentials
$A(x)$ which yield $\pm m$ modes of
magnetic Dirac operators $H_{\!A}$, and 
shall show that these operators $H_{\!A}$ do not have 
$\pm m$ resonances.
Then we shall 
 propose an open question in relation with
  $\pm m$ resonances.


\section{Supersymmetric Dirac operators}  \label{sec:Supersymmetric}
This section is devoted to a discussion about
spectral properties of a  class
of supersymmetric Dirac operators.
We should like to remark that
our approach appears to be in the reverse 
direction in the sense that
 we start with two Hilbert spaces, and introduce
a supersymmetric Dirac operator on
the direct sum of the two Hilbert spaces.
We find this approach convenient for 
our purpose;
see  
Thaller \cite[Chapter 5]{Thaller} for the 
standard theory of the supersymmetric Dirac operator.

The supersymmetric Dirac operator $H$ which we shall consider
in the present paper
is defined as follows:
\begin{equation}   \label{eqn:SSDO-K2-Sec1}
H := 
\begin{pmatrix}
0  &  T^*  \\
T  &  0
\end{pmatrix}
 + 
m  \begin{pmatrix}
I  &  0  \\
0  &  -I
\end{pmatrix} 
\;\; \mbox{ on } \;
\mathcal K = \mathcal H_+  \oplus \mathcal H_-,
\end{equation}
where $T$ is a densely defined 
 operator from a Hilbert space $\mathcal H_+$ to  
another Hilbert space $\mathcal H_-$,
$m$ is a positive constant,
and the identity operators in $\mathcal H_+$ and
$\mathcal H_-$ are both denoted by $I$ with an abuse
of notation. 
We recall that 
 the domain 
of $H$ is given by
 ${\mathcal D}(H)=\mathcal D(T) \oplus \mathcal D(T^*)$,
and that the inner
product of the Hilbert space $\mathcal K$
 is
defined by 
\begin{equation}  \label{eqn:SSDO-K4-Sec1}
(f, \, g)_{\mathcal K} 
 := ( \varphi^+, \, \psi^+)_{\mathcal H_+}
 +
( \varphi^-, \, \psi^-)_{\mathcal H_-}
\end{equation}
for
\begin{equation}    \label{eqn:SSDO-K5-Sec1} 
f= 
\begin{pmatrix}
\varphi^+  \\
\varphi^-
\end{pmatrix}, \quad
g=
\begin{pmatrix}
\psi^+  \\
\psi^-
\end{pmatrix}  \in \mathcal K .
\end{equation}
The first  term
on the right hand side of 
(\ref{eqn:SSDO-K2-Sec1})
is called 
the supercharge and
 the second term
the involution, and denoted by $Q$ and  $\tau$ respectively:
\begin{equation}   \label{eqn:SSDO-K3-Sec1}
Q = 
\begin{pmatrix}
0  &  T^*  \\
T  &  0
\end{pmatrix}, \quad
\tau = 
\begin{pmatrix}
I  &  0  \\
0  &  -I
\end{pmatrix}.
\end{equation}

We now state the main results in
this section, which are about the nature of 
eigenvectors of the supersymmetric Dirac
operator (\ref{eqn:SSDO-K2-Sec1}) at the 
eigenvalues $\pm m$.
We  mention that
$T$ does not need  to be a closed operator,
and that $T^*$ does not need to be densely
defined, 
because we only focus 
 on the eigenvalues $\pm m$ of $H$
and the corresponding eigenspaces.
In the standard theory of 
the supersymmetric Dirac operator, $T$ is assumed to be a densely
defined closed operator and $T^*$ needs to be
densely defined; cf. \cite[\S 5.2.2]{Thaller}.

We  should like to draw attention to the fact
 that
Theorems  \ref{thm:SSDO-main1},
\ref{thm:SSDO-main2}  below
are simply abstract restatements of
Thaller \cite[p. 195, Theorem
7.1]{Thaller}, where he dealt with 
the magnetic Dirac operators under the assumption
that $A_j \in C^{\infty}$. 
 From the mathematically rigorous 
point of view, it is not appropriate to apply 
 \cite[Theorem
7.1]{Thaller} to
the magnetic Dirac operators with 
non-smooth vector potentials. 
However, 
the vector potentials  we shall treat 
in sections \ref{sec:AsymptoticLimits} -- \ref{sec:Structure}
 are not smooth. 
In particular
 we shall deal, in section \ref{sec:Sparseness},
 with vector potentials 
 which can have local singularities.
 In this case, even self-adjointness of the magnetic
 Dirac operators is not trivial.
 Hence
 \cite[Theorem 7.1 ]{Thaller} is not applicable to this case.
These are 
the reasons why we need to generalize 
and restate  \cite[Theorem 7.1 ]{Thaller} 
in an abstract setting.

\vspace{4pt}

\begin{thm} \label{thm:SSDO-main1}
Suppose that $T$ is a densely defined  operator from $\mathcal H_+$ to
$\mathcal H_-$.  Let $H$  be a
supersymmetric Dirac operator defined by {\rm(\ref{eqn:SSDO-K2-Sec1})}. 

\vspace{4pt}

{\rm (i)} If $f= {}^t (\varphi^+ , \, \varphi^-)
\in \mbox{\rm Ker} (H -m)$,
then $\varphi^+ \in \mbox{\rm Ker}(T)$ and $\varphi^-=0$.

\vspace{4pt}

{\rm (ii)} Conversely, if $\varphi^+ \in \mbox{\rm Ker}(T)$, then
$f= {}^t (\varphi^+ , \, 0)
\in \mbox{\rm Ker} (H -m)$.
\end{thm}

\vspace{4pt}

\begin{thm} \label{thm:SSDO-main2}
Assume that $T$ and $H$ are the same
 as in Theorem \ref{thm:SSDO-main1}.

\vspace{4pt}

{\rm (i)} If $f= {}^t (\varphi^+ , \, \varphi^-)
\in \mbox{\rm Ker} (H + m)$,
then $\varphi^+=0$ and $\varphi^- \in \mbox{\rm Ker}(T^*)$.

\vspace{4pt}

{\rm (ii)} Conversely, if $\varphi^- \in \mbox{\rm Ker}(T^*)$, then
$f= {}^t ( 0 , \, \varphi^-)
\in \mbox{\rm Ker} (H + m)$.
\end{thm}

\vspace{4pt}

As immediate consequences, we have 

\vspace{4pt}

\begin{cor}  \label{cor:SSDO-C1}
Assume that $T$ and $H$ are the same
 as in Theorem \ref{thm:SSDO-main1}. Then

\vspace{4pt}

{\rm(i)} $\mbox{\rm Ker}(H-m) = \mbox{\rm Ker}(T)\oplus \{ 0 \}$,
$\mbox{\rm dim} 
 \big( \mbox{\rm Ker} (H - m) \big) 
= 
\mbox{\rm dim} 
 \big( \mbox{\rm Ker} (T) \big)$.

\vspace{4pt}

{\rm(ii)} $\mbox{\rm Ker}(H+m) = \{ 0 \}\oplus\mbox{\rm Ker}(T^*)$,
$\mbox{\rm dim} 
 \big( \mbox{\rm Ker} (H + m) \big) 
= 
\mbox{\rm dim} 
 \big( \mbox{\rm Ker} (T^*) \big)$.
\end{cor}

\vspace{4pt}

The eigenspaces
corresponding to the eigenvalues $\pm m$ of 
supersymmetric Dirac operators do not 
seem to have been 
explicitly formulated 
in the literature as in the form of 
Corollary \ref{cor:SSDO-C1}. 
It is straightforward from this formulation that
the eigenspaces of $H$ corresponding to
the eigenvalue $\pm m$ are independent of $m$.

\vspace{10pt}
\begin{proof}[{\bf Proof of Theorem \ref{thm:SSDO-main1} }]
We first prove assertion (i).  
Let $f= {}^t (\varphi^+ , \, \varphi^-)\in \mbox{\rm Ker} (H -m)$.
We then have
\begin{equation}   \label{eqn:SSDO-K2-pr}
\begin{pmatrix}
0  &  T^*  \\
T  &  0
\end{pmatrix}
\begin{pmatrix}
\varphi^+   \\
\varphi_-
\end{pmatrix}
 + 
m  \begin{pmatrix}
I  &  0  \\
0  &  -I
\end{pmatrix}
\begin{pmatrix}
\varphi^+   \\
\varphi_-
\end{pmatrix}
= m
\begin{pmatrix}
\varphi^+   \\
\varphi_-
\end{pmatrix},
\end{equation}
hence
\begin{equation}   \label{eqn:SSDO-K5}
\begin{cases}
T^* \varphi^- + m \varphi^+
&\!\!\!\!=m \varphi^+  \\
\noalign{\vskip 4pt}
T \varphi^+ - m \varphi^-
 &\!\!\!\! =m \varphi^-, 
\end{cases}
\end{equation}
which immediately implies that $T^*\varphi^-=0$ and $T\varphi^+=2m\varphi^-$.
It follows that
\begin{gather}  \label{eqn:SSDO-K6}
\begin{split}
\Vert T \varphi^+ \Vert^2_{{\mathcal H}_-}  
&= 
(T \varphi^+, \, T \varphi^+)_{{\mathcal H}_-}   \\
&= 
(T \varphi^+, \, 2m \varphi^-)_{{\mathcal H}_-}   \\
&= 
(\varphi^+, \, 2m \,T^* \! \varphi^-)_{{\mathcal H}_+}   \\
&=0.
\end{split}
\end{gather}
Thus we see that $\varphi^+ \in \mbox{Ker}(T)$, and that
$\varphi^-= (2m)^{-1} \, T \varphi^+ = 0$.

We next prove assertion (ii). 
Let $\varphi^+ \in \mbox{\rm Ker}(T)$ and
put $f:= {}^t (\varphi^+ , \, 0)$. 
Then it follows that
$Hf = {}^t (m\varphi^+ , \, T\varphi^+) 
= m \, {}^t (\varphi^+ , \, 0)=mf$.
\end{proof}

\vspace{5pt}

We omit the proof of Theorem \ref{thm:SSDO-main2}, which
 is quite similar to that of 
Theorem \ref{thm:SSDO-main1}.

In connection with our applications to the magnetic
Dirac operator $H_{\!A}$ 
in later sections, we should like to consider
the  case where the Hilbert space $\mathcal H_+$ coincides with
$\mathcal H_-$ and $T$ is self-adjoint ($T^*=T$).
In this case,
the supersymmetric Dirac operator $H$
becomes of the form
\begin{equation}   \label{eqn:SSDO-SA-1}
H = 
\begin{pmatrix}
0  &  T  \\
T  &  0
\end{pmatrix}
 + 
m  \begin{pmatrix}
I  &  0  \\
0  &  -I
\end{pmatrix}
\end{equation}
in the Hilbert space 
${\mathcal K}={\mathcal H}\oplus{\mathcal H}$,
and 
it follows from Theorems
\ref{thm:SSDO-main1} and  
\ref{thm:SSDO-main2} that
the operator $H$ of the form (\ref{eqn:SSDO-SA-1}) possesses of 
a simple but important equivalence:
\begin{gather}  \label{eqn:SSDO-K4}
\begin{split}
T \mbox{ has a zero mode}   
&\Longleftrightarrow  
\; H \mbox{ has an $m$ mode}   \\
{}&\Longleftrightarrow  \; H \mbox{ has a $-m$ mode}, 
\end{split}
\end{gather}
which is actually
a well-known fact: see Thaller \cite[p. 155, Corollary 5.14]{Thaller}.
Here we say that $T$ has a \textit{zero mode} if $0$  is 
 an eigenvalue of $T$.
In a similar manner, we say 
that 
$H$ has an  \textit{$m$ mode} 
(resp. a \textit{$-m$ mode})
if $m$
(resp. $-m$) is  an eigenvalue of $H$.
Furthermore, Theorems
\ref{thm:SSDO-main1} and  
\ref{thm:SSDO-main2} imply the following equivalence for
a zero mode $\varphi$ of $T$:
\begin{gather}  \label{eqn:SSDO-K4+1}
\begin{split}
 T\varphi =0 
\; 
\Longleftrightarrow  
\;  
H
\begin{pmatrix}
\varphi  \\
0  
\end{pmatrix}
=
m
\begin{pmatrix}
\varphi  \\
0  
\end{pmatrix}
\;
\Longleftrightarrow
\;
H
\begin{pmatrix}
0  \\
\varphi    
\end{pmatrix}
=
-m
\begin{pmatrix}
0 \\
\varphi   
\end{pmatrix}.
\end{split}
\end{gather}

We shall show in Theorem \ref{thm:SSDO-main3} below
that  
a sufficient condition for the fact that
\begin{equation}   \label{eqn:rev+1}
\sigma(H) =
\sigma_{\rm{ess}}(H) 
= (-\infty, \, -m] \cup [m, \, \infty)
\end{equation}
is given by the inclusion
$\sigma(T) \supset (0, \, \infty)$.
Therefore
 $\pm m$ are always threshold energies
of  the supersymmetric Dirac operator $H$,
provided that $\sigma(T) \supset (0, \, \infty)$.

\vspace{10pt}
\begin{thm} \label{thm:SSDO-main3}
Let $T$ be a self-adjoint operator
in the Hilbert space $\mathcal H$. Suppose that
$\sigma(T) \supset [0, \, +\infty)$.
Then  
$$
\sigma(H) =(-\infty, \, -m] \cup [m, \, +\infty).
$$
In particular,
$\sigma_d(H)=\emptyset$, i.e., 
the set of discrete eigenvalues of $H$ with finite 
multiplicity is empty.
\end{thm}

\begin{proof}
It follows from (\ref{eqn:SSDO-SA-1}) that
${\mathcal D}(H^2)=\mathcal D(T^2) \oplus \mathcal D(T^2)$
and that
\begin{equation}   \label{eqn:SSDO-SA-2}
H^2=
\begin{pmatrix}
 T^2 + m^2 I & 0 \\
0 & T^2 + m^2 I 
\end{pmatrix} 
\ge m^2 
\begin{pmatrix}
 I & 0 \\
0 &  I 
\end{pmatrix}. 
\end{equation}
This inequality implies that 
$\sigma(H) \subset (-\infty, \, -m] \cup [m, \, +\infty)$.

To complete the proof, 
we shall prove the fact that 
$\sigma(H) \supset (-\infty, \, -m] \cup [m, \, +\infty)$. To this end,
 suppose  $\lambda_0 \in (-\infty,  -m] \cup [m, \, +\infty)$ be given.
Since  $\sqrt{\lambda_0^2 - m^2} \ge 0$,
 we see, by the assumption of the theorem, that
$\sqrt{\lambda_0^2 - m^2} \in \sigma(T)$. 
Therefore, we can find 
a sequence $\{ \psi_n \}_{n=1}^{\infty} \subset \mathcal H$
such that
\begin{equation}     \label{eqn:SSDO-SA-3}
\Vert \psi_n \Vert_{\mathcal H}=1, \;\;
\psi_n \in \mbox{Ran}
\big( 
E_T (\nu_0 - \frac{1}{\,n \,}, \,
     \nu_0  + \frac{1}{\,n \,})
\big),
\quad \nu_0:=\sqrt{\lambda_0^2 - m^2}  
\end{equation}
for each $n$, where $E_T(\cdot)$ is the spectral measure associated 
with $T$:
\begin{equation}    \label{eqn:SSDO-SA-4}
T = \int_{-\infty}^{\infty} \lambda \, dE_T(\lambda).
\end{equation}
Here we have used a basic property of the 
spectral measure: see, for example, Reed and Simon 
\cite[p. 236, Proposition]{ReedSimon}.
It is straightforward to see that
\begin{align}   \label{eqn:SSDO-SA-5}
\Vert (T  - \nu_0) \psi_n \Vert_{\mathcal H} \to 0
\mbox{ as } n \to \infty.   
\end{align}

We shall construct a sequence 
$\{f_n \} \subset \mathcal D(H) =\mathcal D(T)
\oplus \mathcal D(T)$ satisfying 
$\Vert f_n \Vert_{\mathcal K}=1$ and 
$\Vert  (H-\lambda_0) f_n \Vert_{\mathcal K}\to 0$
as $n \to \infty$.
To this end,
we choose a pair of real numbers $a$ and $b$ so
that
\begin{equation}  \label{eqn:SSDO-SA-60}
a^2 + b^2 = 1
\end{equation}
and that
\begin{equation}   \label{eqn:SSDO-SA-61}
\begin{pmatrix}
 m &  \nu_0 \\
 \nu_0 & -m
\end{pmatrix} 
\begin{pmatrix}
 a\\
 b
\end{pmatrix} 
= 
\lambda_0
\begin{pmatrix}
 a\\
 b
\end{pmatrix}. 
\end{equation}
This is possible because the $2\times 2$ symmetric
matrix in (\ref{eqn:SSDO-SA-61}) has 
eigenvalues $\pm \lambda_0$. 
We now put
\begin{equation}      \label{eqn:SSDO-SA-7-1}
f_n :=
\begin{pmatrix}
 a \psi_n \\
 b\psi_n 
\end{pmatrix}\!.
\end{equation}
It is easy to see that  $\Vert f_n \Vert_{\mathcal K} =1$.
By using (\ref{eqn:SSDO-SA-60}) and (\ref{eqn:SSDO-SA-61}),
 we can show that
\begin{align*}   \label{eqn:SSDO-SA-9}
\Vert (H-\lambda_0)f_n \Vert_{\mathcal K}^2
= &
\Vert (m - \lambda_0)a \psi_n + bT \psi_n \Vert_{\mathcal H}^2  \\
&\qquad + \Vert a T\psi_n - (m + \lambda_0) b \psi_n\Vert_{\mathcal H}^2 \\ 
= &
\Vert b( -\nu_0  + T )\psi_n \Vert_{\mathcal H}^2  
 + \Vert a (T- \nu_0) \psi_n\Vert_{\mathcal H}^2 \\  
=&
\Vert (T  - \nu_0) \psi_n \Vert_{\mathcal H}^2 \to 0 
\mbox{ as }  n \to  \infty.
\end{align*}
We thus  have shown 
that $\lambda_0 \in \sigma(H)$.
\end{proof}

Here we briefly mention of
the abstract Fouldy-Wouthuysen transformation
$U_{{}_{FW}}$
in connection with Theorem \ref{thm:SSDO-main3}.
The transformation $U_{{}_{FW}}$
is a unitary  operator in $\mathcal K$,
and  transforms the supersymmetric Dirac operator
$H$ of the form  (\ref{eqn:SSDO-SA-1}) 
into the diagonal form:
\begin{equation*}
U_{{}_{FW}}H  \, U_{{}_{FW}}^*
=
\begin{pmatrix}
\sqrt{T^2 + m^2} & 0  \\
0 & -\sqrt{T^2 + m^2} 
\end{pmatrix}.
\end{equation*}
Note that it is possible
to prove (\ref{eqn:rev+1}) 
based on this unitary equivalence. 
For the abstract Fouldy-Wouthuysen transformation
for the supersymmetric Dirac operator of the
form (\ref{eqn:SSDO-K2-Sec1}),
we refer the reader to 
Thaller \cite[Chapter 5, \S 5.6]{Thaller}.

\vspace{2pt}
In all the later sections,
we shall apply the obtained results on the supersymmetric
Dirac operator to the magnetic Dirac operator $H_{\!A}$ 
of the form (\ref{eqn:1-1}) in the Hilbert space
${\mathcal K} = \big[L^2({\mathbb R}^3)\big]^4$, where we take
$T$ to be the Weyl-Dirac operator
\begin{equation} \label{eqn:2-WD-1}
T_{\!A}= \sigma \cdot (D - A(x) )  
\end{equation}
acting in the Hilbert space 
${\mathcal H}=\big[L^2({\mathbb R}^3)\big]^2$.

As was mentioned above (cf. (\ref{eqn:SSDO-K4}) and (\ref{eqn:SSDO-K4+1})), 
the investigations of properties of $\pm m$ modes of 
the magnetic Dirac operator $H_{\!A}$
are reduced to
the investigations of the corresponding
properties of zero modes of 
the Weyl-Dirac operator $T_{\!A}=\sigma \cdot (D - A(x))$.

We have to  emphasize the broad applicability
of the supersymmetric Dirac operator
in the context of the present paper.
Namely, thanks to the generality of
Theorems
\ref{thm:SSDO-main1} and  
\ref{thm:SSDO-main2}, 
we are able to utilize most of the existing works on
the zero modes of the Weyl-Dirac operator $T_{\!A}$
 (cf.
\cite{AdamMuratoriNash1,AdamMuratoriNash2,AdamMuratoriNash3,
BalinEvan1,BalinEvan2,BalinEvan3,
BalinEvanSaito,
BugliaroFeffermanGraf, Elton, 
ErdosSolovej1, ErdosSolovej2,  ErdosSolovej3,
LossYau})
for the purpose of  
investigating  $\pm m$ modes of 
the magnetic Dirac operator $H_{\!A}$.

\vspace{10pt}

\section{Asymptotic limits of $\pm m$ modes}  
\label{sec:AsymptoticLimits}

In this section,
we consider a class of  magnetic Dirac operators
$H_{\! A}$ under Assumption(SU) below,
and  will focus  on the asymptotic behaviors at infinity of 
$\pm m$ modes of $H_{\! A}$,
assuming that $\pm m$ are the eigenvalues
of $H_{\! A}$.
In section \ref{sec:Concludings},
 we shall see that there exists infinitely many
 $A$'s  such that the corresponding magnetic Dirac 
operators $H_{\! A}$ have the threshold  eigenvalues $\pm m$.

 We now  introduce the terminology
of $\pm m$ modes for the magnetic Dirac operator $H_{\! A}$.

\begin{defn}  \label{defn:Mmodes}
{\rm (Following Lieb \cite{Lieb}.)
By an $m$ mode (resp. a $-m$ mode), 
we mean an eigenfunction corresponding to 
the eigenvalue $m$ (resp. $-m$) of 
$H_{\!A}$, provided that  the threshold energy 
$m$ (resp. $-m$) is 
an eigenvalue of $H_{\!A}$.}
\end{defn}

\vspace{15pt}

\noindent
\textbf{Assumption(SU).} 

\noindent
Each element $A_j(x)$ 
($j=1, \, 2, \, 3$) of $A(x)$ is 
a measurable function satisfying
\begin{equation} \label{eqn:2-1}
| A_j(x) | \le C \langle x \rangle^{-\rho} 
\quad   ( \rho >1 ),
\end{equation}
where $C$ is a positive constant.
\vspace{15pt}

It is easy to see that under Assumption(SU)
the Dirac operator $H_{\! A}$
is a self-adjoint operator in the Hilbert space
 ${\mathcal K}=\big[L^2({\mathbb R}^3)\big]^4$ with  
$\mbox{Dom}(H_{\! A}) = [H^1({\mathbb R}^3)]^4$,
where $H^1({\mathbb R}^3)$ denotes the Sobolev space
of order 1.
Also it is easy to see that under Assumption(SU)
the Weyl-Dirac operator $T_{\! A}$
is a self-adjoint operator in the Hilbert space
 ${\mathcal H}=\big[L^2({\mathbb R}^3)\big]^2$ with  
$\mbox{Dom}(T_{\! A}) = \big[H^1({\mathbb R}^3)\big]^2$.
Since the operator $-\sigma \cdot A(x)$
is relatively compact with respect to the operator 
$T_0:=\sigma \cdot D$, and since $\sigma(T_0)=\mathbb R$,
it follows that  $\sigma(T_A)=\mathbb R$. 
Recalling that
\begin{equation}   \label{eqn:MDO}
H_{\! A} = 
\begin{pmatrix}
0  &  T_A  \\
T_A  &  0
\end{pmatrix}
 + 
m  \begin{pmatrix}
I  &  0  \\
0  &  -I
\end{pmatrix},
\end{equation}
 we can apply Theorem \ref{thm:SSDO-main3}
to  $H_{\! A}$, and get
$$
\sigma(H_{\!A}) = 
\sigma_{\rm ess}(H_{\!A}) = (-\infty, \, -m] \cup  [m, \,+\infty).
$$
Hence $\pm m$ are the threshold energies of 
the operator $H_{\! A}$. 
Assuming that $\pm m$ are the
eigenvalues of
$H_{\! A}$,
we find that
the eigenspaces corresponding to the eigenvalues
$\pm m$ of $H_{\! A}$ are given as 
the direct sum of $\mbox{Ker}(T_A)$ and 
the zero space $\{ 0 \}$ 
(cf. Corollary \ref{cor:SSDO-C1} in section \ref{sec:Supersymmetric}), 
and that these two eigenspaces themselves
as well as their dimensions
are independent of $m$.

\begin{thm} \label{thm:ABEF-main}
Suppose that Assumption{\rm(SU)} is verified, and 
that $m$ {\rm(}resp. $-m${\rm)} is an eigenvalue
 of $H_{\! A}$. 
Let $f$ be an $m$ mode 
{\rm(}resp. a $-m$ modes{\rm)} of $H_{\! A}$.
 Then there exists a zero
mode $\varphi^+$  {\rm(}resp. $\varphi^-${\rm)} of 
$T_A$ such that for any 
$\omega \in {\mathbb S}^2$
\begin{equation}    \label{eqn:limit-1}
\lim_{r \to \infty} r^2 f(r\omega) 
= 
\begin{pmatrix}
u^+(\omega)  \\
0
\end{pmatrix} 
\quad
\big( \mbox{resp.}
\begin{pmatrix}
0  \\
u^-(\omega) 
\end{pmatrix} 
\big),
\end{equation}
where 
\begin{equation}   \label{eqn:lim-1-WD}
u^{\pm}(\omega)  
=
\frac{i}{\, 4\pi \,}
\int_{{\mathbb R}^3}  
\big\{ 
  \big( \omega\cdot A(y) \big)  I_2
    + i\sigma \cdot 
    \big( \omega\times A(y) \big) 
   \big\}  \varphi^{\pm}(y) \, dy,
\end{equation}
and the convergence is uniform with respect to  $\omega$.
\end{thm}

Theorem \ref{thm:ABEF-main}  is a direct consequence
of Corollary \ref{cor:SSDO-C1}, together with 
Sait\={o} and Umeda \cite[Theorem 1.2]{SaitoUmeda1}.
Note that under Assumption(SU) every 
eigenfunction of $H_{\! A}$ corresponding to
either one of eigenvalues $\pm m$ is a continuous   
function of $x$ (cf. Sait\={o} and Umeda \cite[Theorem 2.1]{SaitoUmeda2}),
therefore the expression $f(r\omega)$ in (\ref{eqn:limit-1}) makes sense 
for each $\omega$.

\vspace{10pt}

\section{Sparseness of vector potentials yielding  $\pm m$ modes}   
\label{sec:Sparseness}

In this section, we shall discuss
 the sparseness of the set
of vector potentials $A$ which give rise to 
$\pm m$ modes of magnetic Dirac
operators $H_{\! A}$,
in the sprit of Balinsky and Evans \cite{BalinEvan1} and
\cite{BalinEvan2},
where they investigated Pauli operators
and Weyl-Dirac operators respectively.

We shall make the following assumption:

\vspace{8pt}

\noindent
\textbf{Assumption(BE).}

\noindent 
$A_j\in L^3({\mathbb R}^3)$  for $j=1$, $2$, $3$. 

\vspace{8pt}
Under Assumption(BE) Balinsky and Evans \cite[Lemma 2]{BalinEvan2}
showed that $-\sigma\cdot A$ is infinitesimally small with respect
to $T_0=\sigma\cdot D$ with 
$\mbox{Dom}(T_0)=\big[H^1({\mathbb R}^3)\big]^2$
(see (\ref{eqn:RangeSA-1}) below).
This fact enables us to define the self-adjoint 
realization
$T_A$ in the Hilbert space 
${\mathcal H}=\big[L^2({\mathbb R}^3)\big]^2$
as the operator sum of $T_0$
and $-\sigma\cdot A$, thus  
$\mbox{Dom}(T_A)=\big[H^1({\mathbb R}^3)\big]^2$.
It turns out that
under Assumption(BE), 
 $-\alpha\cdot A$ is infinitesimally small with respect
to $H_0:= \alpha \cdot D + m \beta$, 
and hence we can define the
self-adjoint  realization $H_{\! A}$
in the Hilbert space ${\mathcal K}=\big[L^2({\mathbb R}^3)\big]^4$ 
as the
operator sum of $H_0$ and $-\alpha\cdot A$, thus  
$\mbox{Dom}(H_{\! A})=\big[H^1({\mathbb R}^3)\big]^4$.
Therefore we can regard $H_{\! A}$ as
a supersymmetric Dirac operator, and 
shall apply the results in section 2 to $H_{\! A}$.
(Recall (\ref{eqn:MDO}).)

\vspace{5pt}

\begin{prop}    \label{prop:WD-spec}
Let Assumption\,{\rm(BE)} be satisfied. Then 
$\sigma(T_A)=\mathbb R$.
\end{prop}

We shall prepare a few lemmas for the proof of 
Proposition \ref{prop:WD-spec}.

\begin{lem}   \label{lem:WD-T0}
Let $z \in {\mathbb C} \setminus \mathbb R$.
Then
$\langle D \rangle^{1/2} (T_0 - z )^{-1}$ is a bounded
operator in ${\mathcal H}$. 
Moreover we have
\begin{equation}    \label{eqn:RangeT0}
\mbox{\rm Ran}  \big( \langle D \rangle^{1/2}
 (T_0 - z)^{-1} \big) 
\subset
\big[ H^{1/2}({\mathbb R}^3) \big]^2.
\end{equation}
\end{lem}

\begin{proof}
It is sufficient to show the conclusions of the lemma
for $z=-i$. Let $\varphi \in \mbox{Dom}(T_0)$. 
Then we have
\begin{gather}  \label{eqn:RangeT0-1}
\begin{split}
\Vert (T_0 + i) \varphi \Vert^2_{{\mathcal H}}  
&= 
\int_{{\mathbb R}^3} 
\big| 
\big( (\sigma \cdot \xi) + i I_2 \big)
   \widehat \varphi (\xi)
\big|_{{\mathbb C}^2}^2 \, d\xi   \\
&= 
\int_{{\mathbb R}^3} 
( |\xi|^2 + 1 )
 \big|
\widehat \varphi (\xi)
\big|_{{\mathbb C}^2}^2 \, d\xi    \\
&= 
\Vert \langle D \rangle \varphi \Vert^2_{{\mathcal H}},
\end{split}
\end{gather}
where we have used the anti-commutation relation
$\sigma_j \sigma_k + \sigma_k \sigma_j
 = 2 \delta_{jk}I_2$ in the second equality.
It follows from (\ref{eqn:RangeT0-1})
that
\begin{equation}   \label{eqn:RangeT0-2}
\Vert \varphi \Vert_{{\mathcal H}}  
=\Vert \langle D \rangle (T_0 + i)^{-1}
 \varphi \Vert_{{\mathcal H}} 
\end{equation}
for all $\varphi \in {\mathcal H}$. 
Furthermore, we see that
\begin{gather}  \label{eqn:RangeT0-3}
\begin{split}
\Vert \langle D \rangle^{1/2}
 (T_0 + i)^{-1} \varphi
\Vert_{{\mathcal H}}   
&\le 
 \Vert 
\langle D \rangle^{1/2}  
(T_0 + i)^{-1} \varphi 
     \Vert_{[H^{1/2}({\mathbb R^3})]^2}    \\
&=
\Vert \langle D \rangle  
(T_0 + i)^{-1} \varphi \Vert_{{\mathcal H}}    \\
&= 
\Vert \varphi \Vert_{{\mathcal H}}.
\end{split}
\end{gather}
It is evident that (\ref{eqn:RangeT0-3}) proves
the conclusions of the lemma for $z=-i$.
\end{proof}

\begin{lem}   \label{lem:WD-sigma}
If $\varphi \in \big[ H^{1/2}({\mathbb R}^3) \big]^2$,
then
$(\sigma\cdot A)\langle D \rangle^{-1/2} \varphi \in {\mathcal H}$.
\end{lem}

\begin{proof}
By Balinsky and Evans \cite[Lemma 2]{BalinEvan2}, we see
that for any $\epsilon >0$, there exists
a constant $k_{\epsilon}>0$ such that
for all $\varphi \in \mbox{Dom}(T_0)$ 
\begin{equation}    \label{eqn:RangeSA-1}
\Vert (\sigma \cdot A) \varphi \Vert_{\mathcal H}
\le 
\epsilon  \Vert T_0 \varphi \Vert_{\mathcal H}
 + 
k_{\epsilon }
\Vert \varphi \Vert_{\mathcal H}.
\end{equation}
By virtue of the fact that 
$\langle D \rangle^{-1/2}\varphi \in \mbox{Dom}(T_0)$
for $\varphi \in \big[ H^{1/2}({\mathbb R}^3) \big]^2$,
it follows from (\ref{eqn:RangeSA-1})
that
\begin{align*}    
\Vert (\sigma \cdot A) 
\langle D \rangle^{-1/2} \varphi
\Vert_{\mathcal H}
&\le 
\epsilon  \Vert (T_0 + i) \langle D \rangle^{-1/2}
\varphi \Vert_{\mathcal H}
 + 
k_{\epsilon }
\Vert \langle D \rangle^{-1/2} \varphi 
\Vert_{\mathcal H}                \nonumber\\ 
&\le
\epsilon  \Vert \langle D \rangle^{1/2}
\varphi \Vert_{\mathcal H}
 + 
k_{\epsilon }
\Vert  \varphi \Vert_{\mathcal H} 
                 < + \infty, 
\end{align*}
where we have used (\ref{eqn:RangeT0-1}) 
 and the fact that
$
\Vert \langle D \rangle^{1/2} \varphi \Vert_{\mathcal H}
=
\Vert  \varphi \Vert_{[ H^{1/2}({\mathbb R}^3) ]^2}.
$
\end{proof}

\begin{lem}   \label{lem:WD-sigma1}
$\langle D \rangle^{-1} (\sigma\cdot A)\langle D \rangle^{-1/2}$ 
is a compact operator in ${\mathcal H}$.
\end{lem}

\begin{proof}
One can make a factorization
\begin{gather}   
\begin{split}   \label{eqn:factorization} 
{}&\langle D \rangle^{-1} (\sigma\cdot A)
                   \langle D \rangle^{-1/2}   \\
=&
\Big(
\frac{|D|^{1/2}}{\langle D \rangle} 
\Big)
\Big(
\frac{1}{|D|^{1/2}} (\sigma \cdot A ) \frac{1}{|D|^{1/2}} 
\Big)
\Big(
\frac{|D|^{1/2}}{\langle D \rangle^{1/2}} 
\Big).
\end{split}
\end{gather}
It is obvious that the first term and the last term
on the right hand side of  (\ref{eqn:factorization})
are bounded operators in ${\mathcal H}$.
Then it follows from (\ref{eqn:factorization}) and
Balinsky and Evans \cite[Lemma 1]{BalinEvan2}
that the conclusion of the lemma holds true.
\end{proof}

\begin{lem}   \label{lem:WD-T1}
Let $z \in {\mathbb C} \setminus \mathbb R$.
Then
$(T_A - z)^{-1}\langle D \rangle
\Big|_{ [ H^1({\mathbb R}^3) ]^2}$ 
can be
extended to a bounded operator $\widetilde R_A(z)$ in  ${\mathcal H}$.
Moreover
\begin{equation}    \label{eqn:R-T0}
(T_A - z )^{-1} \varphi = \widetilde R_A(z) 
\langle D \rangle^{-1}
\varphi
\quad {\mbox for \ } \forall \varphi \in {\mathcal H}.
\end{equation}
\end{lem}

\begin{proof}
We first show that $\langle D \rangle(T -z)^{-1}$ is
a closed operator in $\mathcal H$. 
To this end, suppose that $\{ \varphi_j \}$ is a 
sequence in $\mathcal H$ such that
$\varphi_j \to 0$ in $\mathcal H$ and 
$\langle D \rangle (T - z)^{-1}\varphi_j \to \psi$ in 
$\mathcal H$ as $j \to \infty$. 
Then $\{ (T - z)^{-1}\varphi_j \}$ is a Cauchy sequence
in $\big[ H^1({\mathbb R}^3) \big]^2$, hence
there exists a $\widetilde \psi \in \big[ H^1({\mathbb R}^3) \big]^2$
such that
\begin{equation}   \label{eqn:R-T1}
(T - z)^{-1}\varphi_j \to \widetilde \psi
\;\mbox{ in } \big[ H^1({\mathbb R}^3) \big]^2 
\;\mbox{ as } j \to \infty.
\end{equation}
Since the topology of $\big[ H^1({\mathbb R}^3) \big]^2$
is stronger than that of $\mathcal H$,
 (\ref{eqn:R-T1})
implies that
\begin{equation}   \label{eqn:R-T2}
(T - z)^{-1}\varphi_j \to \widetilde \psi
\;\mbox{ in } {\mathcal H}
\;\mbox{ as } j \to \infty.
\end{equation}
On the other hand, since $\varphi_j \to 0$ in $\mathcal H$, and
since $(T - z)^{-1}$ is a bounded operator in $\mathcal H$,
we see that
\begin{equation}   \label{eqn:R-T3}
(T - z)^{-1} \varphi_j \to  0  \; \mbox{ in }  \mathcal H
\end{equation}
as $j \to \infty$. Combining (\ref{eqn:R-T2}) and 
 (\ref{eqn:R-T3}), we see that $\widetilde\psi = 0$.
This fact, together with (\ref{eqn:R-T1}), 
$\langle D \rangle (T - z)^{-1}\varphi_j \to 0$ in 
$\mathcal H$ as $j \to \infty$. Hence $\psi=0$.
We have thus shown that $\langle D \rangle (T - z)^{-1}$
is a closed operator. Noting that 
$\mbox{Dom}(\langle D \rangle (T - z)^{-1})=\mathcal H$,
we can conclude from the Banach closed graph theorem
that  
$\langle D \rangle (T - z)^{-1}$  is a bounded
operator in $\mathcal H$, which will be  
denoted by $Q_A(z)$.

We now put  $\widetilde R_A(z):= Q_A(\overline z)^*$, where 
$Q_A(\overline z)^*$ denotes the adjoint operator of $Q_A(\overline z)$.
Then for any $\varphi \in \mathcal H$ and any
$\psi \in \big[ H^1({\mathbb R}^3) \big]^2$,
we have
\begin{gather}
\begin{split}     \label{eqn:R-T4}
(\varphi, \, \widetilde R_A(z)\psi)_{\mathcal H}
&= 
( Q_A(\overline z) \varphi, \, \psi)_{\mathcal H}     \\
&=
(\langle D \rangle (T - {\overline z})^{-1} \varphi, \, \psi)_{\mathcal H}  
\\ &=
( \varphi, \,  (T - z)^{-1}\langle D \rangle \psi)_{\mathcal H}. 
\end{split}
\end{gather}
It follows from (\ref{eqn:R-T4}) that
\begin{equation}    \label{eqn:R-T5}
\widetilde R_A(z)\psi
=(T - z)^{-1}\langle D \rangle\psi
\end{equation}
for all $\psi \in \big[ H^1({\mathbb R}^3) \big]^2$.
Replacing $\psi$ in (\ref{eqn:R-T5})
with $\langle D \rangle^{-1} \varphi$, 
$\varphi \in \mathcal H$, we get
(\ref{eqn:R-T0}).
\end{proof}

\vspace{5pt}

\begin{proof}[{\bf Proof of Proposition \ref{prop:WD-spec}}]
Since $\sigma(T_0)=\sigma_{\rm ess}(T_0)=\mathbb R$,
it is sufficient to show that 
\begin{equation}   \label{eqn:WD-spec1}
\sigma_{\rm ess}(T_A)=\sigma_{\rm ess}(T_0).
\end{equation}
To this end, we shall prove that the difference
$(T_A +i )^{-1} -  (T_0 +i )^{-1}$
is a compact operator in $\mathcal H$. Then, this fact
implies (\ref{eqn:WD-spec1}); 
see Reed and Simon \cite[p.113, Corollary 1]{ReedSimon4}.

We see that
\begin{align}   
&(T_A +i )^{-1} -  (T_0 +i )^{-1}    \nonumber\\
=& (T_A +i )^{-1} 
    (\sigma \cdot A) (T_0 +i )^{-1} \nonumber \\ 
=& \widetilde R_A(-i)
\{ \langle D \rangle^{-1} (\sigma\cdot A)   
\langle D\rangle^{-1/2}  \}  
\{ \langle D\rangle^{1/2}  (T_0 +i )^{-1} \},   \label{eqn:WD-spec2}
\end{align}
where we have used Lemma \ref{lem:WD-T1} in  (\ref{eqn:WD-spec2}).
It follows from Lemmas \ref{lem:WD-T0}--\ref{lem:WD-T1}
that (\ref{eqn:WD-spec2}) makes sense as a product of
three bounded operators in ${\mathcal H}$ and that 
the product is a compact operator in ${\mathcal H}$. 
\end{proof}

Proposition \ref{prop:WD-spec}, together with Theorem
\ref{thm:SSDO-main3}, gives the following result 
on the spectrum of the magnetic Dirac operator
$H_{\! A}$.

\begin{thm}   \label{thm:BE-1}
Let Assumption\,{\rm(BE)} be satisfied. Then 
$$
\sigma(H_{\! A})=
\sigma_{\rm ess}(H_{\! A})=(-\infty,\, -m] \cup[m, \, \infty).
$$
\end{thm}

\vspace{5pt}

We now state the main results in this section,
which are concerned with the eigenspaces 
corresponding to the threshold eigenvalues 
of the magnetic Dirac operator $H_{\! A}$.

\begin{thm}   \label{thm:BE-2}
Let Assumption\,{\rm(BE)} be satisfied. Then
\vspace{4pt}

{\rm (i)} $\mbox{\rm Ker} (H_{\! A} -m)$ is non-trivial
if and only if 
$\mbox{\rm Ker} (H_{\! A} +m)$ is non-trivial; in other
words,
\begin{gather*}
\begin{split}
{}&\big\{ \,
A \in \big[ L^3({\mathbb R}^3) \big]^3 
\; \big|  \;
\mbox{\rm Ker} (H_{\! A} -m) \not= \{ 0 \} \, \big\}   \\
{}&\quad=
\big\{ \,
A \in \big[ L^3({\mathbb R}^3) \big]^3 
\; \big|  \;
\mbox{\rm Ker} (H_{\! A} + m) \not= \{ 0 \} \, \big\}.
\end{split}
\end{gather*}

\vspace{4pt}

{\rm (ii)} There exists a constant $c$ such that
\begin{equation}  \label{eqn:Daubechies}
\mbox{\rm dim}
\big( \mbox{\rm Ker} (H_{\! A} -m)\big)
=
\mbox{\rm dim}
\big( \mbox{\rm Ker} (H_{\! A} + m)\big)
\le
c \int_{{\mathbb R}^3} |A(x)|^3 \, dx.
\end{equation}
Moreover, the dimension of 
$\mbox{\rm Ker} (H_{\! A} \mp m)$ 
is independent of
$m$. 

\vspace{4pt}

{\rm (iii)} The set 
$\big\{ \,
A \in \big[ L^3({\mathbb R}^3) \big]^3 
\; \big|  \;
\mbox{\rm Ker} (H_{\! A} \mp m) = \{ 0 \} \, \big\}$
contains  an open dense subset of 
$\big[ L^3({\mathbb R}^3) \big]^3$.
\end{thm}

\begin{proof}
By Corollary \ref{cor:SSDO-C1}, we see that
\begin{gather}  \label{eqn:SSDOMagnetic}
\begin{split}
\mbox{Ker}(T_A) \mbox{ is trivial}   
&\Longleftrightarrow  
\; \mbox{\rm Ker} (H_{\! A} - m) \mbox{ is trivial}   \\
{}&\Longleftrightarrow  
\; \mbox{\rm Ker} (H_{\! A} + m) \mbox{ is trivial. } 
\end{split}
\end{gather}
Assertion (i) is equivalent to  (\ref{eqn:SSDOMagnetic}).
Assertion (ii) follows from  Corollary \ref{cor:SSDO-C1}
and Balinsky and Evans \cite[Theorem 3]{BalinEvan2}.
Assertion (iii) follows from  Corollary \ref{cor:SSDO-C1}
and Balinsky and Evans \cite[Theorem 2]{BalinEvan2}.
\end{proof}

\vspace{3pt}

\begin{rem} {\rm
Assertions {\rm(i)} and {\rm(ii)} of Theorem
\ref{thm:BE-2} mean the following facts{\rm:} 
The threshold energy $m$  
is an eigenvalue  of $H_{\! A}$ 
if and only if
the threshold energy $-m$ 
is an eigenvalue of $H_{\! A}$.
If this is the case, their multiplicity are the same.
}
\end{rem}

\vspace{3pt}

\begin{rem}   {\rm
As for the best constant in 
the inequality {\rm(\ref{eqn:Daubechies})},
see Balinsky and Evans \cite[Theorem 3]{BalinEvan2}.
}
\end{rem}

\vspace{10pt}

\section{The structure of the set of vector potentials  
yielding $\pm m$ modes} 
\label{sec:Structure}

In this section, we shall discuss
a property of non-locality of magnetic vector
potentials as well as
 the sparseness of the set
of vector potentials $A$ which give rise to 
$\pm m$ modes of $H_{\! A}$
in the sprit of Elton \cite{Elton},
where he investigated  Weyl-Dirac operators.
We make the following assumption:

\vspace{15pt}

\noindent
\textbf{Assumption(E).}

\noindent 
Each $A_j$ ($j=1$, $2$, $3$) is a real-valued 
continuous function such that $A_j(x) = o(|x|^{-1})$
as $|x| \to \infty$.

\vspace{10pt}
It is straightforward to see that under Assumption(E), 
$-\sigma\cdot A$ is a bounded self-adjoint operator
in the Hilbert space 
$\mathcal H=\big[L^2({\mathbb R}^3)\big]^2$.
Hence we can define the self-adjoint
realization $T_A$ 
with $\mbox{Dom}(T_A)=\big[H^1({\mathbb R}^3)\big]^2$
as the
operator sum of $T_0$ and $-\sigma\cdot A$.

Also, it is straightforward to see that
 $-\alpha\cdot A$ is 
a bounded self-adjoint operator
in the Hilbert space 
${\mathcal K}=\big[L^2({\mathbb R}^3)\big]^4$,
hence we can define the
self-adjoint  realization $H_{\! A}$
with 
$\mbox{Dom}(H_{\! A})=\big[H^1({\mathbb R}^3)\big]^4$ in
${\mathcal K}$ as the operator sum of $H_0$ and
$-\alpha\cdot A$. Therefore, in the same way as in
section 5,
 we can regard $H_{\! A}$ as
a supersymmetric Dirac operator, and 
 apply the results in section 2
to $H_{\! A}$.

We note that under Assumption(E),
$(-\sigma\cdot A)(T_0 + i)^{-1}$ is a compact
operator in  
$\mathcal H$.
Hence, 
in the same way as in the proof
of Proposition \ref{prop:WD-spec}, 
we can show that
$\sigma(T_A) =\mathbb R$. 
This fact, together with Theorem \ref{thm:SSDO-main3},
implies the following result.

\begin{thm}    \label{thm:decay-1}
Let Assumption\,{\rm(E)} be satisfied. Then 
$$
\sigma(H_{\! A})=
\sigma_{\rm ess}(H_{\! A})=(-\infty,\, -m] \cup[m, \, \infty).
$$
\end{thm}

To state the main results in this section,
we need to introduce the following notation:
\begin{equation}    \label{eqn:eltonAssump}
{\mathcal A}:=
\{ \, A \, | \, A \mbox{ satisfies Assumption(E)} \, \}.
\end{equation}
We regard $\mathcal A$ as a Banach space
with the norm
$$
\Vert A \Vert_{\mathcal A}
= \sup_{x} \{ \langle x \rangle |A(x)| \}
$$

\vspace{1pt}

\begin{thm} \label{thm:elton-2}
Let Assumption{\rm(E)} be satisfied. Define
$$
{\mathcal Z}_k^{\pm} = 
\{ \, A \in \mathcal A \, | \,
  \mbox{\rm dim} (\mbox{\rm Ker}(H\mp m) ) = k
\, \}
$$
for $k=0$, $1$, $2$, $\cdots$. Then

\vspace{4pt}

{\rm (i)} ${\mathcal Z}_k^{+}={\mathcal Z}_k^{-}$ for all $k$.

\vspace{4pt}

{\rm (ii)} ${\mathcal Z}_0^{\pm}$ is an open and dense 
subset of $\mathcal A$.

\vspace{4pt}

{\rm(iii)} For any $k$ and any open subset
$\Omega (\not= \emptyset)$ of ${\mathbb R}^3$
there exists an  $A\in {\mathcal Z}_k^{\pm}$ 
such that $A \in 
\big[ C_0^{\infty}(\Omega) \big]^3$.
\end{thm}

\begin{proof}
Assertion (i) is a direct consequence of 
Corollary \ref{cor:SSDO-C1}.
Assertions (ii) and (iii) follows from
Corollary \ref{cor:SSDO-C1} and 
Elton \cite[Theorem 1]{Elton}.
\end{proof}

It is of some interest to point out a conclusion
following from Theorem \ref{thm:ABEF-main} and 
Assertion (iii) of Theorem \ref{thm:elton-2}.
Namely, there are (at least) countably infinite number of
vector potentials $A$ with compact support
such that the corresponding Dirac operators
$H_{\! A}$ have $\pm m$ modes
$f^{\pm}$ 
with  the property (\ref{eqn:lim-1-WD}).
The $\pm m$ modes $f^{\pm}$ behave like 
$|f^{\pm}(x)| \asymp |x|^{-2}$ for $|x| \to \infty$,
in spite of the fact that the vector potentials
and the corresponding magnetic fields vanish outside
 bounded regions.  It is obvious that
this phenomenon describes a
certain kind of non-locality.

Also, it is of some interest to mention that
$H_{\! A}$ does not have $\pm m$ resonances
if the support of vector potential $A$ is compact.
This is an immediate consequence of Theorem \ref{thm:NER-1}
in the next section.

\vspace{10pt}

\section{Non existence of $\pm m$ resonances}  \label{sec:Resonances}

In this section, we will work in bigger Hilbert spaces than
${\mathcal H}= [L^2({\mathbb R}^3)]^2$ 
and ${\mathcal K}=[L^2({\mathbb R}^3)]^4$.
Therefore, the results on the supersymmetric Dirac operators
in section \ref{sec:Supersymmetric} are not applicable
in this section.

In this section, 
we shall occasionally write the inner product  of ${\mathcal H}$
as
\begin{equation*}
    (\varphi, \psi)_{{\mathcal H}}
 = \int_{{\mathbb R}^3} 
(  \varphi(x), \, \psi(x)  )_{{\mathbb C}^2} dx
\end{equation*}
for $\varphi$, $\psi\in{\mathcal H}$,
 where 
$( \cdot \, ,  \cdot)_{{}_{{\mathbb C}^2}}$ denotes 
 the inner product of
${\mathbb C}^2$.

 We need to introduce weighted $L^2$ spaces in order to
 deal with $\pm m$ resonances,     which do not belong
 to the Hilbert space $\mathcal K$.    
By $L^{2, s}({\mathbb R}^3)$, we mean the weighted $L^2$ space
defined by
\begin{equation*}
 L^{2, s}({\mathbb R}^3)
  :=\{ \, u \; | \; \langle x \rangle^s u 
 \in L^2({\mathbb R}^3) \, \} \qquad (s \in \mathbb R)
\end{equation*}
where  $\langle x \rangle = \sqrt{1 + |x|^2 \,}$, and
we set
\begin{equation*}
   {\mathcal L}^{2,s} = [L^{2,s}({\mathbb R}^3)]^4.
\end{equation*}

\vspace{6pt}

\begin{defn}  \label{defn:Mresonance}
{\rm
 By an $m$ resonance (resp. a $-m$ resonance), 
we mean a function 
$f \in {\mathcal L}^{2, -s}\setminus {\mathcal K}$,
$0< s \le 3/2$,   
such that
$H_{\!A}f = mf$ 
(resp. $H_{\!A}f = -mf$ )
in the distributional sense.}
\end{defn}

\vspace{10pt}

We would like to caution that in Definition \ref{defn:Mresonance}
 one has to take the meaning of
$H_{\!A}f =\pm mf$ in the distributional sense,
because of the reason that
 $\pm m$ resonances 
do not belong to the Hilbert space ${\mathcal K}$,
hence do not belong to the domain of the self-adjoint 
realization of $H_{\!A}$.
For this reason, we let
$H_{\!A}$ stand for the formal differential operator
throughout this section,
in spite of the fact that 
$H_A$ has the unique self-adjoint realization in
$\mathcal K$ under the assumption of Theorem \ref{thm:NER-1}
below.
We hope this will not cause any confusion.

\vspace{8pt}

\begin{thm} \label{thm:NER-1}
Assume that
each element $A_j(x)$ 
{\rm(}$j=1, \, 2, \, 3${\rm)} of $A(x)$ is 
a measurable function satisfying
\begin{equation} \label{eqn:R-1}
| A_j(x) | \le C \langle x \rangle^{-\rho} 
\quad   ( \rho >3/2 ),
\end{equation}
where $C$ is a positive constant.
Suppose that
$f = {}^t (\varphi^+ , \, \varphi^-)$
belongs to ${\mathcal  L}^{2,-s}$  for some $s$
with $0 < s < \min (1, \, \rho -1)$ and  
 satisfies $H_A f=mf$ 
{\rm(}resp. $H_A f= -mf \,${\rm)} in the distributional sense.
Then  $f \in [H^1({\mathbb R}^3)]^4$ and $\varphi^- =0$
{\rm(}resp. $\varphi^+ =0${\rm)}.
\end{thm}

Theorem \ref{thm:NER-1} implies the non-existence
of $\pm m$ resonances in the sense of
Definition \ref{defn:Mresonance}, 
as well as in the sense described
in the following theorem.

\begin{thm} \label{thm:NER-2}
Let $A(x)$ satisfy the same assumption as in Theorem \ref{thm:NER-1}.
Suppose that
$f$
belongs to
$ [L^2_{\rm loc}({\mathbb R}^3)]^4$ and 
satisfies either equation of 
$H_A f=\pm m f$ in the distributional sense.
In addition, suppose that 
$f$ has the asymptotic expansion
\begin{equation}   \label{eqn:R-2}
f(x) = C_1 |x|^{-1} + C_2 |x|^{-2} + o(|x|^{-2})
\end{equation}
as $|x| \to \infty$, where $C_1$ and $C_2$ are
constant vectors in ${\mathbb C}^4$.
Then  $C_1 =0$.
\end{thm}

\begin{proof}
It follows from (\ref{eqn:R-2}) that
$f \in {\mathcal L}^{2,-s}$  for any $s$
with $1/2 < s < 1$. This fact, together with the 
assumptions of the theorem, enables us 
to apply Theorem \ref{thm:NER-1} and to conclude that
$f \in [H^1({\mathbb R}^3)]^4$. In particular, 
$f \in {\mathcal K}$, which leads to the fact that $C_1=0$.
\end{proof}

We shall give  a proof of Theorem \ref{thm:NER-1} 
only for $m$ resonances,
since the proof for $-m$ resonances is similar.
Roughly speaking, we will mimick the idea of the proof of 
assertion (i) of Theorem \ref{thm:SSDO-main1}.
Therefore
we need 
the Weyl-Dirac operator
$T_{\!A}=\sigma \cdot (D - A(x))$ again.
However, we are not allowed to use the Weyl-Dirac operator as 
a self-adjoint operator in the Hilbert 
space ${\mathcal H}$,
but only allowed to use it as a formal 
differential operator instead.
This is because
$\pm m$ resonances do not belong to  ${\mathcal K}$.
This fact 
causes  complication, in a certain extent, in the proof
of 
 Theorem \ref{thm:NER-1}.

We begin the proof of Theorem \ref{thm:NER-1}
with a lemma whose proof will
be given after the proof of the theorem.
The proof  of the lemma is lengthy.

\vspace{6pt}

\begin{lem} \label{lem:resonlem1}
Under the hypotheses of Theorem \ref{thm:NER-1}, 
$\varphi^{\pm}$ have the following properties{\rm:}
\vspace{4pt}

{\rm (i)} $(\sigma \cdot D)\varphi^+ \in {\mathcal H}$, 
 $(\sigma \cdot A)\varphi^+\in
{\mathcal H}$ and
   $\varphi^- \in [H^1({\mathbb R}^3)]^2$.

\vspace{4pt}

{\rm (ii)}  
$\big( (\sigma \cdot D)\varphi^+, \, \varphi^- \big)_{\mathcal H}
=\big( (\sigma \cdot A)\varphi^+, \, \varphi^- \big)_{\mathcal H}$.
\end{lem}

\vspace{10pt}
\begin{proof}[{\bf Proof of Theorem \ref{thm:NER-1}}]
Let $f$ satisfy $H_Af = mf$ in the distributional sense.
Then we have
\begin{align}   \label{eqn:resonLem-1}
\begin{cases}
 m \varphi^+ + \sigma \cdot (D-A(x)) \varphi^-
&\!\!\!\!=m \varphi^+  \\
\noalign{\vskip 4pt}
\sigma \cdot (D-A(x)) \varphi^+ - m \varphi^-
&\!\!\!\!  =m \varphi^- 
\end{cases}
\end{align}
 in the distributional sense, which immediately implies 
\begin{equation}   \label{eqn:resonLem-2}
\sigma \cdot (D-A(x)) \varphi^-=0 
\end{equation}
and
\begin{equation}   \label{eqn:resonLem-3}
\sigma \cdot (D-A(x)) \varphi^+  =2m \varphi^-. 
\end{equation}
In view of Lemma \ref{lem:resonlem1}, 
it follows from (\ref{eqn:resonLem-3}) that
\begin{gather}
\begin{split}
4m^2 \Vert \varphi^- \Vert_{\mathcal H}^2
&= 2m \big( 2m \varphi^-, \, \varphi^-\big)_{\mathcal H}  \\
&= 2 m 
\big( \sigma \cdot (D-A) \varphi^+, \, \varphi^-\big)_{\mathcal H}  \\
&=2m \big\{
\big( (\sigma \cdot D) \varphi^+, \, \varphi^-\big)_{\mathcal H}  
 - \big( (\sigma \cdot A) \varphi^+, \, \varphi^-\big)_{\mathcal H} 
\big\} \\
&=0.
\end{split}
\end{gather}
Hence $\varphi^-=0$. This fact, together with (\ref{eqn:resonLem-3}),
means that 
\begin{equation} \label{eqn:resonLem-4}
\sigma \cdot (D-A(x)) \varphi^+  =0
\end{equation}
in  the distributional sense.  
It follows from Sait\={o} and Umeda \cite[Theorem 2.2]{SaitoUmeda2}
that $\varphi^+ \in [H^1({\mathbb R}^3)]^2$. 
(Note that
the hypothesis $0 < s < \min (1, \, \rho -1)$ is stronger
than the one imposed in \cite[Theorem 2.2]{SaitoUmeda2}.)
This implies that $f \in [H^1({\mathbb R}^3)]^4$, 
because $\varphi^-=0$ as was shown above.
\end{proof}

\vspace{10pt}

Before proving Lemma \ref{lem:resonlem1},
we should like to remark that (\ref{eqn:resonLem-2})
and (\ref{eqn:resonLem-3}) follow directly
from the hypothesis that $H_A f = m f$
in the distributional sense.
Therefore we are allowed to use 
(\ref{eqn:resonLem-2})
and (\ref{eqn:resonLem-3}) in the proof of 
Lemma \ref{lem:resonlem1} below.

\vspace{10pt}
\begin{proof}[{\bf Proof of Lemma \ref{lem:resonlem1}}]
Since $\rho - s >1$ by assumption, we see that 
\begin{equation}  \label{eqn:resonLem-5}
(\sigma \cdot A) \varphi^+ \in 
[L^{2, \, \rho - s}({\mathbb R}^3)]^2 \subset {\mathcal H}.
\end{equation}
It follows from (\ref{eqn:resonLem-2}) 
and \cite[Theorem 2.2]{SaitoUmeda2}
that $\varphi^- \in [H^1({\mathbb R}^3)]^2$.
This fact, together with (\ref{eqn:resonLem-3})
and (\ref{eqn:resonLem-5}), 
implies that 
$(\sigma \cdot D)\varphi^+ 
= 2m\varphi^- + (\sigma \cdot A) \varphi^+ 
\in
\mathcal H$. Thus assertion (i) is proved.

In order to prove assertion (ii), we need to
introduce a cutoff function. Let $\chi$ be a
function in $C^{\infty}(\mathbb R)$
 such that $0 \le \chi \le 1$, $\chi(r) = 1$ ($r \le 1$),
and $\chi(r) = 0$ ($r \ge 2$). Set 
\begin{equation} \label{eqn:resonLem-(ii)1}
\chi_n(x) = \chi(|x|/n) \qquad (n=1, \,2, \, 3, \, \cdots ).
\end{equation}
It is evident that
\begin{equation}  \label{eqn:resonLem-(ii)2}
\big( (\sigma \cdot D)\varphi^+, \, \varphi^- \big)_{\mathcal H}
= \lim_{n \to \infty}
\big( (\sigma \cdot D)\varphi^+, \, \chi_n\varphi^- \big)_{\mathcal H}.
\end{equation}

Let $\{j_{\varepsilon}\}_{0<\varepsilon <1}$ be Friedrichs' mollifier,
 {\it i.e.},
$j_{\epsilon}(x):= {\varepsilon}^{-3}j(x/\varepsilon)$,
where $j \in C_0^{\infty}({\mathbb R}^3)$ and 
$\Vert j \Vert_{L^1}=1$.
Since $\chi_n\varphi^- \in \mathcal H$, we see that
$j_{\varepsilon}\ast(\chi_n\varphi^-)$ converges to
 $\chi_n\varphi^-$ in $\mathcal H$ as $\varepsilon \downarrow 0$.
Hence, for each $n$, we have
\begin{equation}  \label{eqn:resonLem-(ii)3}
\big( (\sigma \cdot D)\varphi^+, \, \chi_n\varphi^- \big)_{\mathcal H}
= \lim_{\varepsilon \downarrow 0}
\big( (\sigma \cdot D)\varphi^+, \, 
j_{\varepsilon}\ast(\chi_n\varphi^-) \big)_{\mathcal H}.
\end{equation}
It is straightforward that 
$j_{\varepsilon}\ast(\chi_n\varphi^-) \in [C_0^{\infty}({\mathbb R}^3)]^2$
and that
\begin{equation}  \label{eqn:resonLem-(ii)4}
\mbox{supp}[j_{\varepsilon}\ast(\chi_n\varphi^-)] 
\subset 
\big\{ \, x \, \big| \, |x| \le 2n +1 \, \big\}.
\end{equation}
Appealing to the definition of the distributional derivatives,
we get
\begin{equation}   \label{eqn:resonLem-(ii)5}
\big( (\sigma \cdot D)\varphi^+, \, 
j_{\varepsilon}\ast(\chi_n\varphi^-) \big)_{\mathcal H}
= 
\int_{{\mathbb R}^3} 
\big( 
\varphi^+(x), \, 
(\sigma \cdot D)(j_{\varepsilon}\ast(\chi_n\varphi^-)) (x)
\big)_{{\mathbb C}^2} \, dx.  
\end{equation}
For each $n$ and $\varepsilon$, we have
\begin{gather}   \label{eqn:resonLem-(ii)6}
\begin{split}
(\sigma \cdot D)(j_{\varepsilon}\ast(\chi_n\varphi^-)) (x)
&=
\int_{{\mathbb R}^3} 
(\sigma \cdot D_x) \big( j_{\varepsilon}(x - y) \big) \,
 \chi_n(y) \varphi^-(y) \, dy  \\
&=
\int_{{\mathbb R}^3} 
- (\sigma \cdot D_y) \big( j_{\varepsilon}(x - y)  \big) \,
 \chi_n(y) \varphi^-(y) \, dy  \\
&=
\int_{{\mathbb R}^3} 
j_{\varepsilon}(x - y)  \,
 (\sigma \cdot D_y) \big( \chi_n(y) \varphi^-(y) \big)   \, dy   \\
\noalign{\vskip 4pt}
&=
j_{\varepsilon} \ast
\big\{
 (\sigma \cdot D) ( \chi_n \varphi^-)
\big\} (x).
\end{split}
\end{gather}
In the third equality of (\ref{eqn:resonLem-(ii)6})
we have regarded 
$j_{\varepsilon}(x -\cdot)$ as a function in $C_0^{\infty}({\mathbb R}_y^3)$ 
and have appealed to the definition 
of the destributional derivatives with respect
to $y$ variable.
Note that
\begin{equation}   \label{eqn:resonLem-(ii)7}
(\sigma \cdot D) ( \chi_n \varphi^-)
= \{ (\sigma \cdot D) \chi_n \} \varphi^-
+   \chi_n (\sigma \cdot D) \varphi^-.
\end{equation}
Combining (\ref{eqn:resonLem-(ii)5}), (\ref{eqn:resonLem-(ii)6})
and (\ref{eqn:resonLem-(ii)7}), we obtain
\begin{gather}   \label{eqn:resonLem-(ii)8}
\begin{split}
\big( (\sigma \cdot D)\varphi^+, \, 
j_{\varepsilon}\ast(\chi_n\varphi^-) \big)_{\mathcal H}
&= 
\int_{{\mathbb R}^3} 
\big( 
\varphi^+(x), \, 
j_{\varepsilon}\ast \big[ \{(\sigma \cdot D)\chi_n \}\,\varphi^-\big]
(x)
\big)_{{\mathbb C}^2} \, dx    \\
&\quad +
\int_{{\mathbb R}^3} 
\big( 
\varphi^+(x), \, 
j_{\varepsilon}\ast \big[\chi_n \, (\sigma \cdot D)\varphi^- \big] (x)
\big)_{{\mathbb C}^2} \, dx
\end{split}
\end{gather}
We examine the limit of each integral on the
right hand side of (\ref{eqn:resonLem-(ii)8})
as $\varepsilon \downarrow 0$.
As for the first integral, we have
\begin{gather}   \label{eqn:resonLem-(ii)9}
\begin{split}
{}&\qquad \Big|
\int_{{\mathbb R}^3} 
\big( 
\varphi^+(x), \, 
j_{\varepsilon} \ast  \big[ \{(\sigma \cdot D)\chi_n \}\,\varphi^-\big]
(x)
\big)_{{\mathbb C}^2} \, dx    \\
& \qquad \qquad \qquad \qquad \qquad   -
\int_{{\mathbb R}^3} 
\big( 
\varphi^+(x), \, 
 \{(\sigma \cdot D)\chi_n \}(x)\,\varphi^-(x)
\big)_{{\mathbb C}^2} \, dx  \,
\Big|  \\
{}&\le
\int_{|x|\le 2n +1} 
\big| \varphi^+(x) \big|_{{\mathbb C}^2} \\
{}& \qquad \qquad \times \big| j_{\varepsilon} \ast  
\big[ \{(\sigma \cdot D)\chi_n \}\,\varphi^-\big](x)
-
\{(\sigma \cdot D)\chi_n \}(x)\,\varphi^-(x)  \big|_{{\mathbb C}^2} \, dx  \\
\noalign{\vskip 7pt}
{}&\le
\big\Vert
\big| \varphi^+ \big|_{{\mathbb C}^2}
\big\Vert_{L^2(|x|\le 2n+1)} \;
\big\Vert
j_{\varepsilon} \ast  
\big[ \{(\sigma \cdot D)\chi_n \}\,\varphi^-\big]
-
\{(\sigma \cdot D)\chi_n \}\,\varphi^-
\big\Vert_{\mathcal H}   \\
\noalign{\vskip 4pt}
{}& \quad \to 0  \quad (\varepsilon \downarrow 0),
\end{split}
\end{gather}
since $\{(\sigma \cdot D)\chi_n \}\,\varphi^- \in \mathcal H$.
In the first inequality (\ref{eqn:resonLem-(ii)9}) we have 
used the Schwarz inequality in ${\mathbb C}^2$,
and 
in the second inequality the Schwarz inequality in $L^2$.
Therefore
\begin{gather}   \label{eqn:resonLem-(ii)10}
\begin{split}
{}&\lim_{\varepsilon \downarrow 0}
\int_{{\mathbb R}^3} 
\big( 
\varphi^+(x), \, 
j_{\varepsilon}\ast \big[ \{(\sigma \cdot D)\chi_n \}\,\varphi^-\big]
(x)
\big)_{{\mathbb C}^2} \, dx  \\
{}&\qquad\quad  =
\int_{{\mathbb R}^3} 
\big( 
\varphi^+(x), \, 
  \{(\sigma \cdot D)\chi_n \}(x) \,\varphi^-(x)
\big)_{{\mathbb C}^2} \, dx.
\end{split}
\end{gather}
In a similar manner, we see that
\begin{gather}   \label{eqn:resonLem-(ii)11}
\begin{split}
{}&\lim_{\varepsilon \downarrow 0}
\int_{{\mathbb R}^3} 
\big( 
\varphi^+(x), \, 
j_{\varepsilon}\ast \big[ \chi_n \,(\sigma \cdot D)\varphi^-\big](x)
\big)_{{\mathbb C}^2} \, dx  \\
{}&\qquad\quad  =
\int_{{\mathbb R}^3} 
\big( 
\varphi^+(x), \, 
  \chi_n (x) \, (\sigma \cdot D)\varphi^-(x)
\big)_{{\mathbb C}^2} \, dx,
\end{split}
\end{gather}
where we have used the fact 
that $\varphi^- \in [H^1({\mathbb R}^3)]^2$. (Recall that
this fact was shown in assertion (i) of the lemma.)
It follows from (\ref{eqn:resonLem-(ii)3}), 
(\ref{eqn:resonLem-(ii)8}),
(\ref{eqn:resonLem-(ii)10}) and  (\ref{eqn:resonLem-(ii)11})
that
\begin{gather}   \label{eqn:resonLem-(ii)12}
\begin{split}
\big( 
(\sigma \cdot D)\varphi^+, \, \chi_n\varphi^- 
   \big)_{\mathcal H}  
& =
\int_{{\mathbb R}^3} 
\big( 
\varphi^+(x), \, 
  \{(\sigma \cdot D)\chi_n \}(x) \,\varphi^-(x)
\big)_{{\mathbb C}^2} \, dx    \\
\noalign{\vskip 4pt}
{}&\qquad  +
\int_{{\mathbb R}^3} 
\big( 
\varphi^+(x), \, 
  \chi_n (x) \, (\sigma \cdot D)\varphi^-(x)
\big)_{{\mathbb C}^2} \, dx.
\end{split}
\end{gather}

To estimate the first integral on the right hand side
of (\ref{eqn:resonLem-(ii)12}), we need the fact 
that
\begin{equation}   \label{eqn:resonLem-(ii)13}
\{(\sigma \cdot D)\chi_n \}(x) 
 = \frac{1}{\,n \,} \, 
\chi^{\prime}\Big(\frac{|x|}{n} \Big) \,
\frac{1}{i} (\sigma \cdot  \omega) 
\qquad (\,\omega=x/|x| \,).
\end{equation}
Note that $\mbox{supp}[(\sigma \cdot D)\chi_n ]
\subset \big\{ \, x \, \big| \, n \le |x| \le 2n \, \big\}$
and that $\sigma\cdot \omega$ is a unitary matrix.
Hence we have
\begin{gather}   \label{eqn:resonLem-(ii)14}
\begin{split}
{}&\left|
\int_{{\mathbb R}^3} 
\big( 
\varphi^+(x), \, 
  \{(\sigma \cdot D)\chi_n \}(x) \,\varphi^-(x)
\big)_{{\mathbb C}^2} \, dx \, 
\right|  \\
\noalign{\vskip 4pt}
{}& \quad \le
\frac{1}{\, n \,} 
\Big( \sup_{r>0}|\chi^{\prime}(r)| \Big)
\int_{n \le |x| \le 2n}
\big\vert \varphi^+(x) \big\vert_{{\mathbb C}^2}
\,
\big\vert \varphi^-(x) \big\vert_{{\mathbb C}^2} \, dx \\
\noalign{\vskip 4pt}
& \quad \le
\frac{1}{\, n \,} 
\Big( \sup_{r>0}|\chi^{\prime}(r)| \Big) \,
\Big\{
\int_{n \le |x| \le 2n}
\langle x \rangle^{-2s}
\big\vert \varphi^+(x) \big\vert^2_{{\mathbb C}^2} 
\, dx \Big\}^{\!1/2}    \\
\noalign{\vskip 2pt}
{}& \qquad \qquad \times
\Big\{
\int_{n \le |x| \le 2n}
\langle x \rangle^{2s}
\big\vert \varphi^-(x) \big\vert^2_{{\mathbb C}^2} 
\, dx \Big\}^{\!1/2}       \\
\noalign{\vskip 4pt}
& \quad \le
\frac{1}{\, n \,}   
\Big( \sup_{r>0}|\chi^{\prime}(r)| \Big) \,
\big\Vert  
\big| \varphi^+ \big\vert_{{\mathbb C}^2} 
\big\Vert_{L^{2, \, -s} }
\times
(1 + 4n^2)^{s/2} 
\big\Vert  \varphi^- \big\Vert_{\mathcal H}   \\
\noalign{\vskip 4pt}
& \quad \le
\mbox{const.} \,
n^{-1 +s}  \,
\big\Vert  
\big| \varphi^+ \big\vert_{{\mathbb C}^2} 
\big\Vert_{L^{2, \, -s} } \,
\big\Vert  \varphi^- \big\Vert_{\mathcal H}   \\
\noalign{\vskip 4pt}
{}& \qquad  \to 0  \quad ( \, n \to \infty ),
\end{split}
\end{gather}
since $s < 1$ by assumption of Theorem \ref{thm:NER-1}.
Thus the first integral on the right hand side
of (\ref{eqn:resonLem-(ii)12}) tends to $0$ 
as $n \to \infty$.

We now investigate the limit of 
the second integral on the right hand side
of (\ref{eqn:resonLem-(ii)12}) 
as $n \to \infty$.
It follows from (\ref{eqn:resonLem-2}) that
\begin{gather}  \label{eqn:resonLem-(ii)15}
\begin{split}
{}&
\int_{{\mathbb R}^3} 
\big( 
\varphi^+(x), \, 
  \chi_n(x) \,(\sigma \cdot D)\varphi^-(x)
\big)_{{\mathbb C}^2} \, dx    \\
& \qquad \qquad \qquad \qquad \qquad   -
\int_{{\mathbb R}^3} 
\big( 
(\sigma \cdot A)(x) \, \varphi^+(x), \, 
  \varphi^-(x)
\big)_{{\mathbb C}^2} \, dx  \,   \\
\noalign{\vskip 4pt}
{}& \quad =
\int_{{\mathbb R}^3} 
\big( 
 \big( \chi_n(x) -1 \big)
  \,(\sigma \cdot A)(x) \varphi^+(x), \, 
 \varphi^-(x) 
\big)_{{\mathbb C}^2} \, dx  , 
\end{split}
\end{gather}
where we have used the fact
that $(\sigma \cdot A)(x)$ is a Hermitian matrix 
for each $x$.
Noting (\ref{eqn:resonLem-5}), we find that
the absolute value of the right hand side of
(\ref{eqn:resonLem-(ii)15})
 is less than or equal to
\begin{equation}  
\big\Vert  
(\chi_n -1) (\sigma \cdot A) \varphi^+ 
\big\Vert_{\mathcal H} \,
\big\Vert 
 \varphi^- 
\big\Vert_{\mathcal H},  
\end{equation}
which obviously tends to $0$ as $n\to \infty$. 
Combining this  fact with (\ref{eqn:resonLem-(ii)12}),
 (\ref{eqn:resonLem-(ii)14}) and
 (\ref{eqn:resonLem-(ii)15}),
we obtain
\begin{equation}  \label{eqn:resonLem-(ii)17}
\lim_{n \to \infty} 
\big( 
(\sigma \cdot D)\varphi^+, \, \chi_n\varphi^- 
   \big)_{\mathcal H} 
=
\big( 
(\sigma \cdot A)\varphi^+, \, \varphi^- 
   \big)_{\mathcal H} .
\end{equation}
Assertion (ii) of the lemma is a direct
consequence of (\ref{eqn:resonLem-(ii)2})
and (\ref{eqn:resonLem-(ii)17}).
\end{proof}

\vspace{10pt}

\section{Exapmles, concluding remarks and an open question}  
\label{sec:Concludings}

We shall give examples of vector potentials
$A(x)$ which yield $\pm m$ modes but
do not give rise to $\pm m$ resonances. 
The basic idea in this section is to
exploit the equivalences (\ref{eqn:SSDO-K4}), 
(\ref{eqn:SSDO-K4+1}), and
 to apply Theorem \ref{thm:NER-1}.
 It turns out that
 beautiful spectral properties are
in common to all the examples of the magnetic 
Dirac operators in this
section. See properties
(i) -- (iv) of Examples \ref{expl:LY} and \ref{expl:AMN}.

\vspace{10pt}

\begin{expl}[\textbf{Loss-Yau}]  \label{expl:LY}
{\rm
Let 
\begin{align}      \label{eqn:LY-1}
A_{{}_{LY}}(x) = 3 \langle x \rangle^{-4} 
  \big\{ (1-|x|^2)  w_0 + 2 (w_0 \cdot x) x + 2 w_0 \times x \big\}
\end{align}
where
$\langle x \rangle = \sqrt{1 + |x|^2 \,}$,
$\phi_0 = {}^t(1, \, 0)$ {\rm (}$\phi_0$ can
be any unit vector in ${\mathbb C}^2${\rm )}, 
and
\begin{equation}      \label{eqn:LY-1-1}
w_0   =  \phi_0 \cdot (\sigma \phi_0 )
:=\big( 
(\phi_0, \, \sigma_1 \phi_0)_{\!_{{\mathbb C}^2}}, \, 
  (\phi_0, \, \sigma_2 \phi_0)_{\!_{{\mathbb C}^2}}, \,
(\phi_0, \, \sigma_3 \phi_0)_{\!_{{\mathbb C}^2}}
 \big).
\end{equation}
Here $w_0 \cdot x$ and $w_0 \times x$
denotes the inner product  and 
the exterior 
product of
$\,{\mathbb R}^3$ respectively.
Then the magnetic Dirac operator
$$
H_{{}_{LY}}:=H_{\!A_{LY}} = \alpha \cdot (D - A_{{}_{LY}}(x) ) + m
\beta
$$ 
has  the following properties:
\begin{itemize}
\item[(i)] $\sigma(H_{{}_{LY}})=
   \sigma_{\rm ess}(H_{{}_{LY}})=
(-\infty, \, -m] \cup [m, \,\infty)${\rm;} 
\item[(ii)] $H_{{}_{LY}}$ has $\pm m$ modes. Moreover, the point spectrum of $H_{{}_{LY}}$ consists only of $\pm m$,
i.e. $\sigma_p(H_{{}_{LY}})=\{ -m,\, m\}${\rm;} 
\item[(iii)] $H_{{}_{LY}}$ does not have $\pm m$ resonances{\rm;}
\item[(iv)] $H_{{}_{LY}}$ is absolutely continuous on 
$(-\infty, \, -m) \cup (m,\, \infty)$.
\end{itemize}

We shall show these properties one-by-one.
It is easy to see that $-\sigma \cdot
A_{{}_{LY}}(x)$  is  relatively compact perturbation of  $T_0=\sigma \cdot D$,
hence the Weyl-Dirac operator 
$$
T_{{}_{LY}}:= T_{\!A_{LY}} =\sigma \cdot (D - A_{{}_{LY}}(x))
$$
is a self-adjoint
operator in the Hilbert space
${\mathcal H}=\big[L^2({\mathbb R}^3)\big]^2$ with 
the domain $\big[H^1({\mathbb R}^3)\big]^2$. 
Since the spectrum of the operator $T_0$
equals the whole real line,
we see that $\sigma(T_{{}_{LY}})=\mathbb R$. Property {\rm(i)} immediately follows from Theorem
\ref{thm:SSDO-main3}.

We shall show property {\rm(ii)}. 
According to Loss and Yau \cite[section II]{LossYau}, 
the Weyl-Dirac operator 
$T_{{}_{LY}}$
has a zero mode $\varphi_{{}_{LY}}$ defined by
\begin{equation}    \label{eqn:LY-2}
\varphi_{{}_{LY}}(x) = \langle x \rangle^{-3}
 \big( I_2 + i \sigma \cdot x \big) \phi_0. 
\end{equation}
It follows from  {\rm(\ref{eqn:SSDO-K4})} and
{\rm(\ref{eqn:SSDO-K4+1})}
 that ${}^t(\varphi_{{}_{LY}}, \, 0 )$ 
{\rm(}resp. ${}^t(0, \,\varphi_{{}_{LY}}) )$ is 
an $m$ mode {\rm(}resp. a $-m$ mode{\rm)} of $H_{{}_{LY}}$. 
Hence $\sigma_p(H_{{}_{LY}}) \supset \{-m, \, m\}$.
On the other hand, it follows from 
Yamada \cite{Yamada-1} 
that $H_{{}_{LY}}$ has no eigenvalue in 
$(-\infty, -m) \cup (m, \,\infty)$.
{\rm(}Note that the vector potential $A_{{}_{LY}}$ 
satisfies the assumption of 
\cite[Proposition 2.5]{Yamada-1}.{\rm)}
This fact, together with property {\rm(i)}, implies that
 $\sigma_p(H_{{}_{LY}}) \subset \{-m, \, m\}$. Summing up, 
we get property {\rm(ii)}.
Since $|A_{{}_{LY}}(x)|\le C \langle x \rangle^{-2}$, 
property {\rm(iii)} follows from Theorem \ref{thm:NER-1}.
Property {\rm(iv)} is a direct consequence of 
Yamada \cite[Corollary 4.2]{Yamada-1}.
As for absolutely continuity and 
limiting absorption principle for
Dirac operators, see also Yamada \cite{Yamada-2},
Balslev and Helffer \cite{BalslevHellfer}, and
Pladdy, Sait{\=o} and Umeda \cite{PladdySaitoUmeda}.
}
\end{expl}

 \vspace{3pt}

\begin{rem}{\rm
Since $A_{{}_{LY}}$ is $C^{\infty}$, one can apply 
Thaller \cite[p. 195, Theorem
7.1]{Thaller} to conclude that 
${}^t(\varphi_{{}_{LY}}, \, 0 )$ 
{\rm(}resp. ${}^t(0, \,\varphi_{{}_{LY}}) )$ is 
an $m$ mode {\rm(}resp. a $-m$ mode{\rm)} of $H_{{}_{LY}}$.
 This fact
is also mentioned in Thaller \cite{Thaller1}.
}
\end{rem}

 \vspace{3pt}

\begin{rem}{\rm
As was pointed out in Loss and Yau \cite[section II]{LossYau},
one sees that $\mbox{div}  A_{{}_{LY}}\not=0$, and
one can find,
by a gauge transformation,
 a vector potential $\tilde{A}_{{}_{LY}}$
which satisfies $\mbox{div} \tilde{A}_{{}_{LY}}=0$ and
$\mbox{rot} \tilde{A}_{{}_{LY}} = \mbox{rot} A_{{}_{LY}}$
and yields a  zero mode $\tilde{\varphi}_{{}_{LY}}$
of $\sigma \cdot (D - \tilde{A}_{{}_{LY}}(x))$.
In fact, defining 
\begin{equation*}
\tilde{A}_{{}_{LY}}:= A_{{}_{LY}} + \nabla \chi_{{}_{LY}},
\quad
\tilde{\varphi}_{{}_{LY}}:= e^{i\chi_{{}_{LY}}}\varphi_{{}_{LY}}
\end{equation*}
with 
\begin{equation*}
\chi_{{}_{LY}}(x) := \frac{1}{\, 4\pi \,}
\int_{{\mathbb R}^3} \frac{1}{\, |x-y|\, }
(\mbox{div} A_{{}_{LY}})(y) \, dy,
\end{equation*}
we observe that $\tilde{A}_{{}_{LY}}$ and $\tilde{\varphi}_{{}_{LY}}$
have the desired properties mentioned above.
Moreover, we can show that
$|\tilde{A}_{{}_{LY}}(x)| \le C \langle x \rangle^{-2} 
(\in L^6({\mathbb R}^3))$. 
Hence, the magnetic Dirac operator 
$H_{\!\tilde{A}_{LY}} = \alpha \cdot (D - \tilde{A}_{{}_{LY}}(x) ) + m$
shares the properties (i) - (iv) of Example \ref{expl:LY} 
with $H_{\!A_{LY}} = \alpha \cdot (D - \tilde{A}_{{}_{LY}}(x) ) + m$.
This same idea is applicable to
the vector potentials $A^{(\ell)}$ 
in Example \ref{expl:AMN} below.
}
\end{rem}

\vspace{3pt}

\begin{expl}[\textbf{Adam-Muratori-Nash}]  \label{expl:AMN}
{\rm
In the same spirit as in Example \ref{expl:LY},
we can show the existence of countably infinite
number of vector potentials with which the magnetic Dirac
operators have the properties {\rm (i) -- (iv)} in
Example \ref{expl:LY}.

In fact,
we shall exploit a result on the Weyl-Dirac 
operator by Adam, Muratori and Nash \cite{AdamMuratoriNash1},
where they construct a series 
 of 
vector potentials $A^{(\ell)}$
$(\ell = 0$, $1$, $2$, $\cdots)$, each of  which gives rise a zero
mode $\psi^{(\ell)}$ of the Weyl-Dirac operator 
$T^{(\ell)}:= \sigma \cdot (D - A^{(\ell)}(x) )$.
The idea of  \cite{AdamMuratoriNash1} is an extension  of 
that of  Loss and Yau \,\cite[section II]{LossYau}{\rm;} Indeed 
$A^{(0)}$ and $\psi^{(0)}$ give the same
vector potential and zero mode as in
{\rm(\ref{eqn:LY-1})} and  {\rm(\ref{eqn:LY-2})}.
For $\ell \ge 1$, the construction of the zero mode
$\psi^{(\ell)}(x)$ is based on an anzatz {\rm(}see {\rm(7)} in section II of
{\rm \cite{AdamMuratoriNash1})} and 
the definition of $A^{(\ell)}$  is given by
\begin{align}      
A^{(\ell)}(x) =
\frac{h^{(\ell)}(x)}{| \psi^{(\ell)}(x)|^2}
\{ \psi^{(\ell)}(x) \cdot (\sigma \psi^{(\ell)}(x)) \},  \label{eqn:AMN-2}
\end{align}
where $h^{(\ell)}(x)$ is a real valued function defined
as
\begin{equation}
h^{(\ell)}(x)= \frac{c_{\ell}}{\langle x \rangle^2 }  \quad (c_{\ell} \;
\mbox{  a real constant depending only on } \ell)
\end{equation}
and 
$\psi^{(\ell)}(x) \cdot (\sigma \psi^{(\ell)}(x))$ is 
defined in the same way as in {\rm(\ref{eqn:LY-1-1})}.
(For the definition of $h^{(\ell)}(x)$, see Sait\={o} 
and Umeda \cite{SaitoUmeda3}.)
By the same arguments as in 
Example \ref{expl:LY}, we can deduce that
 the magnetic Dirac operator 
$H^{(\ell)}:= \alpha \cdot (D - A^{(\ell)}(x) ) + m
\beta$, $\;\ell = 0$, $1$, $2$, $\cdots$, has 
 the properties  {\rm (i) -- (iv)}
of
Example \ref{expl:LY}.
}
\end{expl}

\vspace{5pt}

Section \ref{sec:AsymptoticLimits} was based upon our results on
supersymmetric Dirac operators in section \ref{sec:Supersymmetric} 
of the
present paper
 and those of
Sait\=o and Umeda 
\cite{SaitoUmeda1}.
It turned out that all $\pm m$ mode have the 
same asymptotic limit at infinity, i.e. $\asymp |x|^{-2}$
as $|x| \to \infty$. This means that
the asymptotic limits of $\pm m$ modes
of the mangetic Dirac operator 
are the same as those of zero modes of
the Weyl-Dirac operator.
Section \ref{sec:Sparseness}
 was based upon our results on supersymmetric
Dirac operators in section \ref{sec:Supersymmetric} 
 and those of Balinsky and Evans \cite{BalinEvan2}
 on the Weyl-Dirac operator.
Section \ref{sec:Structure} 
was based upon our results on supersymmetric
Dirac operators in section \ref{sec:Supersymmetric}
 and those of Elton \cite{Elton} on the Weyl-Dirac operator.
 In each section from sections \ref{sec:AsymptoticLimits} to  
\ref{sec:Structure}, 
we  made a different assumption on 
the vector potentials. It is meaningful to
compare these assumptions with each other.
To this end, imitating (\ref{eqn:eltonAssump}),
we introduce
the following notation
\begin{align*}   
{\mathcal A}_{SU}:=&
\{ \, A \, | \, A \mbox{ satisfies Assumption(SU)} \, \},  \\
{\mathcal A}_{BE}:=&
\{ \, A \, | \, A \mbox{ satisfies Assumption(BE)} \, \}.
\end{align*}
We then have
\begin{align*}
&\qquad\quad {\mathcal A}_{SU} \subsetneqq {\mathcal A}_{BE},   
 \\
&{\mathcal A}\setminus{\mathcal A}_{SU}\not=\emptyset,
\quad
{\mathcal A}_{SU}\setminus{\mathcal A}\not=\emptyset,
 \\ 
&{\mathcal A}\setminus{\mathcal A}_{BE}\not=\emptyset,
\quad
{\mathcal A}_{BE} \setminus {\mathcal A} \not=\emptyset.
\end{align*}
In sections \ref{sec:Sparseness} and \ref{sec:Structure}, 
it was shown that
the set of vector potentials which give rise to
$\pm m$ modes is scarce in  each regime. 
The non-existence of $\pm m$ resonances
was proved in section \ref{sec:Resonances} 
under the assumption that 
$|A_j(x)| \le C \langle x \rangle^{-\rho}$, $\rho > 3/2$.
Based on the results in section \ref{sec:Resonances} ,
it follows that all the examples of vector potentials 
in this section do not have $\pm m$ resonances.
A natural question arises:

\begin{center}
\textit{Is there 
a vector potential $A$ which satisfies 
$|A_j(x)| \le C \langle x \rangle^{-\rho}$, \\ 
$\rho > 0$, 
and yields $\pm m$ resonances 
of the magnetic Dirac operator $H_{\! A}$?} 
\end{center}


\vspace{15pt}

\textbf{Acknowledgments.} The authors would like to thank 
the referees for their valuable comments and constructive
suggestions
which led them to the consideration of 
threshold resonances and had them improve
 the paper.

\vspace{10pt}

\end{document}